\documentclass{amsart}

\usepackage{amsfonts}
\usepackage{epstopdf}
\usepackage{amssymb,amsmath}
\usepackage{graphicx}
\usepackage{bm,bbm}
\usepackage{color}
\usepackage{tikz}
\usepackage[table]{xcolor}
\usepackage{subcaption}
\usepackage{ifthen}
\usepackage{float}
\usepackage{verbatim}
\usepackage{pgfplots}
\usepackage{enumitem}
\usepackage{array,blkarray}
\usepackage{bigdelim}
\usepackage{hyperref}
\usepackage{listings} 
\usepackage{booktabs}

\usepackage{isotope}
\usepackage{stmaryrd}
\usepackage[]{placeins}
\usepackage{algorithm}
\usepackage{algpseudocode}
\pgfplotsset{compat=1.18} 

\usetikzlibrary{arrows.meta}
\usetikzlibrary{backgrounds}
\usepgfplotslibrary{patchplots}
\usepgfplotslibrary{fillbetween}

\ifpdf
  \DeclareGraphicsExtensions{.eps,.pdf,.png,.jpg}
\else
  \DeclareGraphicsExtensions{.eps}
\fi

\newtheorem{theorem}{Theorem}
\newtheorem{proposition}[theorem]{Proposition}
\newtheorem{corollary}[theorem]{Corollary}
\theoremstyle{definition}
\newtheorem{definition}[theorem]{Definition}

\newtheorem{lemma}[theorem]{Lemma}

\theoremstyle{remark}
\newtheorem{remark}[theorem]{Remark}

\numberwithin{theorem}{section}
\numberwithin{equation}{section}
\numberwithin{table}{section}
\numberwithin{figure}{section}
\definecolor{lgray}{RGB}{200,200,200} 
\newcommand\C{\mathbb C}
\newcommand\D{\mathbb D}
\newcommand\N{\mathbb N}
\newcommand\R{\mathbb R}
\newcommand{\SSS}{\mathbb{S}}
\newcommand\Sphere{\mathcal{S}}

\DeclareSymbolFont{stmry}{U}{stmry}{m}{n}

\DeclareMathAlphabet{\mathpzc}{OT1}{pzc}{m}{it}

\newcommand\OB{\mathcal{O}\mathcal{B}_N^{\,\mathbb{C}}(p,\Hsp)}
\newcommand\DpC{\mathbb{D}(p,\C)}
\newcommand\DpR{\mathbb{D}(p,\R)}
\newcommand\DpS{\mathbb{D}(p,\SSS)}
\newcommand\Seg{\mathcal{M}}

\newcommand\Horg[2]{\mathcal{H}_{#1}^{#2}}

\newcommand{\DC}{\mathcal{D},\mathbb{C}}

\newcommand{\re}{\mbox{\rm Re}}

\newcommand{\sR}{\mbox{\rm \tiny R}}
\newcommand{\nablaR}{\nabla^{\hspace{0pt}\sR}}
\newcommand{\VR}{V^{\hspace{0pt}\sR}}

\DeclareMathOperator*{\argmin}{arg\,min}
\DeclareMathOperator*\myinf{\vphantom{p}inf}

\DeclareMathOperator{\diag}{diag}
\DeclareMathOperator{\Diag}{Diag}

\DeclareMathOperator{\dist}{dist}
\DeclareMathOperator{\grad}{grad}
\DeclareMathOperator{\Drm}{D}
\newcommand{\spot}{\mathcal{L}_3}
\DeclareMathOperator{\trace}{trace}
\newcommand{\hrm}{{\rm h}}

\newcommand\calA{\mathcal A}
\newcommand\calB{\mathcal B}

\newcommand\calD{\mathcal D}
\newcommand\calE{\mathcal E}
\newcommand\calF{\mathcal F}
\newcommand\calG{\mathcal G}
\newcommand\calH{\mathcal H}
\newcommand\calI{\mathcal I}

\newcommand\calL{\mathcal L}
\newcommand\calM{\mathcal M}
\newcommand\calN{\mathcal N}
\newcommand\calO{\mathcal O}
\newcommand\calP{\mathcal P}
\newcommand\calR{\mathcal R}

\newcommand\calS{\mathcal S}

\newcommand\calV{\mathcal V}

\newcommand\DDEjj{\Drm^2_{\varphi_j\varphi_j}\!\calE}
\newcommand\DDLjj{\Drm^2_{\varphi_j\varphi_j}\!\calL}
\newcommand{\ci}{\mathrm{i}} 

\newcommand{\gbf}{\bm{g}}

\newcommand{\ybf}{\bm{y}}

\newcommand{\ubf}{\bm{u}}
\newcommand{\vbf}{\bm{v}}
\newcommand{\wbf}{\bm{w}}

\newcommand{\zbf}{\bm{z}}

\newcommand{\nubf}{{\bm \nu}}

\newcommand{\xibf}{{\bm \xi}}

\newcommand{\zetabf}{{\bm \zeta}}
\newcommand{\varthetabf}{{\bm \vartheta}}
\newcommand{\varphibf}{{\bm \varphi}}
\newcommand{\psibf}{{\bm \psi}}
\newcommand{\zerobf}{\bm{0}}

\newcommand{\V}{V}

\newcommand{\Hsp}{H}
\newcommand{\Lsp}{L}
\def\drm{\rm d}

\def\dx{\,\text{d}x}

\makeatletter
\newsavebox{\@brx}
\newcommand{\llangle}[1][]{\savebox{\@brx}{\(\m@th{#1\langle}\)}%
  \mathopen{\copy\@brx\mkern2mu\kern-0.9\wd\@brx\usebox{\@brx}}}
\newcommand{\rrangle}[1][]{\savebox{\@brx}{\(\m@th{#1\rangle}\)}%
  \mathclose{\copy\@brx\mkern2mu\kern-0.9\wd\@brx\usebox{\@brx}}}
\makeatother

\newcommand{\coout}[2]{\llangle{#1},{#2}\rrangle}
\newcommand{\out}[2]{\llbracket{#1},{#2}\rrbracket}

\newcommand{\proofbox}{\vbox{\hrule height0.6pt\hbox{\vrule height1.3ex width0.6pt\hskip0.8ex\vrule width0.6pt}\hrule height0.6pt}}
 
\hypersetup{
    colorlinks=true,
    linkcolor=blue,
    filecolor=blue,      
    urlcolor=black,
    citecolor= blue
}

\title[Riemannian Optimization for Rotating Multicomponent BEC]{Qualitative and Quantitative Analysis of Riemannian Optimization Methods for Ground States of Rotating Multicomponent Bose-Einstein Condensates}

\author[M. Hermann]{Martin Hermann}
\address[M. Hermann]{Institut f\"ur Mathematik, Universit\"at Augsburg, Universit\"atsstra{\ss}e~12a, 86159 Augsburg, Germany}
\email{martin.hermann@uni-a.de}

\author[T. Stykel]{Tatjana Stykel}
\address[T. Stykel]{Institut f\"ur Mathematik \& Centre for Advanced Analytics and Predictive Sciences (CAAPS), Universit\"at Augsburg, Universit\"atsstra{\ss}e~12a, 86159 Augsburg, Germany}
\email{tatjana.stykel@uni-a.de}

\author[M. Yadav]{Mahima Yadav}
\address[M. Yadav]{Fakultät für Mathematik, Ruhr-Universität Bochum, Universitätsstraße 150, 44780 Bochum, Germany}
\email{mahima.yadav@rub.de}

\thanks{The work of M.~Hermann and T.~Stykel is part of a project that has received funding from the German Research Foundation – Project number 564828373.}
\begin{document}
\begin{abstract}
We develop and analyze Riemannian optimization methods for computing ground states of rotating multicomponent Bose-Einstein condensates, defined as minimizers of the Gross– Pitaevskii energy functional. To resolve the non-uniqueness of ground states induced by phase invariance, we work on a quotient manifold endowed with a general Riemannian metric. By introducing an~auxiliary phase-aligned iteration and employing fixed-point convergence theory, we establish a unified local convergence framework for Riemannian gradient descent methods and derive explicit convergence rates. Specializing this framework to two metrics tailored to the energy landscape, we study the energy-adaptive and Lagrangian-based Riemannian gradient descent methods. While monotone energy decay and global convergence are established only for the former, a~quantified local convergence analysis is provided for both methods. Numerical experiments confirm the theoretical results and demonstrate that the Lagrangian-based method, which incorporates second-order information on the energy functional and mass constraints, achieves faster local convergence than the energy-adaptive scheme. 
\end{abstract}
\maketitle

%
%
%
\section{Introduction}\label{sec:introduction}
In Bose–Einstein condensates (BECs), quantized vortices are a distinctive signature of superfluidity that appear when the condensate is subjected to rotation. A particularly interesting scenario arises when different species of particles or different hyperfine states of the same type of particles coherently fuse into a~single quantum state, called a~\emph{multicomponent} BEC. The first experiments demonstrating the formation of two-component BECs have been reported in~\cite{MyaBGCW97,SteISMCK98}.  In this paper, we are concerned with the numerical approximation of ground states of {\it rotating} condensates comprising multiple interacting components, each representing a~different atomic or molecular species.

For a rotating $p$-component BEC confined to a~domain $\calD\subset \R^d$ with $d=2,3$, every quantum state $\varphibf = (\varphi_1, \ldots, \varphi_p) \in [H^1_0(\mathcal{D}, \mathbb{C})]^p$ can be associated with an~energy which depends on trapping potentials $V_1, \ldots,V_p$, 
rotational frequencies $\Omega_1, \ldots, \Omega_p \in \mathbb{R}$, and interaction parameters $\kappa_{ij} \in \R$. Accordingly, the Gross-Pitaevskii energy functional $\calE: [H^1_0(\mathcal{D}, \mathbb{C}]^p \rightarrow \mathbb{R}$ has the form
\begin{equation}\label{eq:energy}
\calE(\varphibf) 
= \sum_{j=1}^p \int_\calD \frac{1}{2}\, \|\nabla\varphi_j\|^2 
+ \frac{1}{2}\, \V_j(x)\, |\varphi_j|^2  - \frac{1}{2}\Omega_j\overline{\varphi}_j\spot \varphi_j
+ \frac{1}{4}\, \rho_j(\varphibf) \,|\varphi_j|^2 \, \dx,
\end{equation}
 where $\spot=-\ci (x_1\partial_{x_2}-x_2\partial_{x_1})$ denotes the $x_3$-component of the angular momentum, $\ci=\sqrt{-1}$, and
\[
\rho_j(\varphibf)=\sum_{i=1}^p \kappa_{ij} |\varphi_i|^2, \qquad j=1,\ldots,p,
\]
are linear combinations of the density functions $|\varphi_i|^2$. The parameters $\kappa_{ij}$ reflect the nature and strength of particle interactions between the $i$-th and $j$-th condensate components. These interactions can be either attractive ($\kappa_{ij} <0$) or repulsive ($\kappa_{ij} >0$). The \emph{ground states} of the system are defined as global minimizers of the energy functional $\,\calE$ subject to mass constraints
\begin{equation}\label{eq:mass}
\|\varphi_j\|_{L^2(\mathcal{D},\C)}^2
    := \int_\calD |\varphi_j|^2\,\dx
    = N_j, \qquad j=1,\ldots,p,
\end{equation}
which ensure that each condensate component contains a~prescribed number of particles $N_j$, with the total number of particles given by $N_1+\ldots+N_p$. The ground states represent the most stable stationary configurations of the multicomponent BEC system. 
  
The numerical approximation of ground states relies on two fundamental aspects: the construction of a discrete space and the choice of an optimization method. Although spatial discretization plays a~critical role, the present work focuses on the approximation of ground states from the perspective of optimization, rather than discussing the discretization techniques. In the existing literature, a wide range of numerical methods have been developed to compute ground states in the single-component setting. These approaches are generally built on either minimizing the associated energy functional or exploiting the eigenvalue formulation of the problem. Among the prominent techniques are gradient flows~\cite{BaoD04,BWM05,CDLX23}, 
Riemannian optimization methods in discrete and continuous settings~\cite{AHYY26,AltPS24,AntLT17,DanP17}, and
variants of Sobolev gradient methods, including standard and projected versions~\cite{AntLT17,ChenLLZ24,DanK10,DanP17,HenP20,HenY25,KaE10}.

Non-rotating multicomponent BECs have been extensively studied in \cite{AHPS25,Bao04,BaoC13,CalORT09,HuaY24}. However, including the rotation 
terms into these systems introduces significant analytical challenges, most notably the transition from a~purely real-valued framework to a~complex-valued one, and the resulting non-uniqueness of ground states arising from phase and rotational symmetries. Although rare, numerical studies of the single-component case (cf. \cite[Sect.~6.2]{BWM05}) suggest that multiple ground states with different numbers of vortices may exist at certain critical frequencies. While rotating condensates composed of a~single species have been thoroughly investigated in \cite{FenT25,HHSW24,HenY25}, multicomponent BECs in the rotating regime have received comparatively less attention \cite{AntD14, ZhaDTDCZ14}. The analysis of such complex systems becomes considerably more involved, both analytically and numerically, primarily due to the intricate interplay between rotational effects and inter-component interactions. 

To address the non-uniqueness of ground states stemming from the phase invariance of the Gross-Pitaevskii energy functional~$\,\calE$, we develop a~theoretical framework which is based on the construction of quotient spaces, obtained by identifying elements that are equivalent under the action of the rotational symmetry group. By formulating the minimization problem on the corresponding quotient manifold, we effectively eliminate the degeneracy induced by rotational symmetry. Endowing this manifold with a~general Riemannian metric, we develop a Riemannian gradient descent method (RGD) in a~general setting. For metrics satisfying certain natural structural conditions, we establish a local convergence of this method and derive explicit linear contraction rates for the iterations. The proof carefully accounts for the phase invariance of the energy functional and utilizes the non-singularity of the Lagrangian Hessian on the horizontal space at a ground state, ensured by its local quasi-uniqueness, known also as the Morse–Bott condition~\cite{FenT25}. For quantitative analysis, we introduce an~auxiliary iteration that artificially adjusts the phase at each step to ensure that the iterates remain aligned with the phase of the limiting ground state. The local convergence of these auxiliary iterates can then be analyzed using an~abstract fixed point convergence theory, in particular Ostrowski’s theorem \cite{Ost66, Shi81}, and transferred back to the original RGD scheme. 

Building on this general result, we then specialize to two particular choices of metrics: the energy-adaptive metric and the Lagrangian-based metric. The energy-adaptive Riemannian metric, first introduced in \cite{HenP20}, is designed to reflect the geometry induced by the underlying energy functional, thereby enhancing the stability and efficiency of the energy-adaptive Riemannian gradient descent method (eaRGD). This scheme ensures monotone energy decay and global convergence of the iterates, while our general contraction framework yields a quantitative result on local linear convergence near a~ground state under the assumption of local quasi-uniqueness. We further extend the Lagrangian-based Riemannian gradient descent method (LagrRGD), developed in~\cite{AHPS25} for non-rotating multicomponent condensates, to the case of rotating systems. Our analysis provides a~rigorous local linear convergence result for LagrRGD in both rotating and non-rotating multicomponent settings. Thanks to the incorporation of second-order information on the energy functional and mass constraints, LagrRGD exhibits superior local convergence properties compared to eaRGD, as confirmed by numerical experiments. 

The paper is organized as follows.
In Section~\ref{sec:setting}, we introduce the functional analytical framework for the constrained energy minimization problem and establish the existence of a ground state. 
In Section~\ref{sec:derivatives}, we analyze the properties of the first- and second-order derivatives of the energy functional and relate the energy minimization problem to a~nonlinear eigenvector problem. 
Section~\ref{sec:geometry} focuses on the geometric framework, exploring the decomposition of the tangent space of the gene\-ralized oblique manifold.
Section~\ref{sec:uniqueness} reformulates the minimization problem on the quotient manifold and addresses the local quasi-uniqueness of ground states. 
In Section~\ref{sec:RGD_loc}, we present the RGD method induced by a~general Riemannian metric and establish its local convergence properties. By introducing the energy-adaptive and Lagrangian-based  metrics, we also analyze the associated RGD schemes. 
Section~\ref{sec:numerics} contains results from numerical experiments that validate the theoretical findings. Concluding remarks are presented in Section~\ref{sec:conclusion}.

\textbf{Notation.}
We denote by $\R$ the set of real numbers, by $\C$ the set of complex numbers, and by $\SSS=\{e^{\ci\alpha}, \,\alpha\in \R\}$ the unit circle in $\C$. 
For $p\in\mathbb{N}$ and $S\in\{\R,\C,\SSS\}$, a~set of $p\times p$ diagonal matrices with diagonal entries in~$S$ is denoted by $\D(p,S)$. 
The trace of a matrix $M\in\C^{p\times p}$ is denoted by $\trace M$, and $\diag(v)$ stands for the diagonal matrix with components of $v\in\C^p$ on the diagonal. In addition, we denote by $I_p$ and $0_p$ the $p\times p$ identity and zero matrices, respectively. The Euclidean vector norm is denoted by $\|\cdot\|$, and the spectral matrix norm is denoted by $\|\cdot\|_2$. 
%
%
\section{Constrained energy minimization and existence of a ground state} \label{sec:setting}
In this section, we formulate the constrained energy minimization problem characterizing a~ground state of rotating multicomponent BECs and establish its existence under appropriate assumptions on the model parameters.

%
%
\subsection{Mathematical model}\label{sec:setting:energy}
Let $\mathcal{D} \subset\R^d$ with $d=2,3$ be a~bounded convex Lipschitz domain. We consider the Lebesgue space $L^2(\mathcal{D},\C)$ and the Sobolev space $H_0^1(\mathcal{D}, \C)$ over the field of real numbers $\R$  with the following real inner products 
\[
(v, w)_{L^2} 
    = \re \big( \int_{\mathcal{D}} v \, \overline{w} \,\dx \big) 
\qquad \text{and} \qquad
(v, w)_{H_0^1} 
    = \re \big( \int_{\mathcal{D}} \nabla v \cdot  \overline{\nabla w} \,\dx \big),
\]  
 respectively. The consideration of the real spaces is motivated by the fact that the Gross-Pitaevskii energy functional $\,\calE$ in \eqref{eq:energy} evaluated at complex-valued wave functions takes real values. Occasionally, we also need the complex $L^2$-inner product $(v, w)_{L^2_{\C}} = \int_{\mathcal{D}} v \, \overline{w} \,\dx$. For $p\geq 1$, we define the Hilbert spaces $\Lsp=[L^2(\mathcal{D},\C)]^p$ and $\Hsp=[H_0^1(\mathcal{D},\C)]^p$ of $p$-frames. These spaces form a~Gelfand triple $\Hsp \subset \Lsp \subset \Hsp^\star$, where~$\Hsp^\star=[H^{-1}(\mathcal{D},\C)]^p$ is the dual space of $\Hsp$. 
On the pivot space $\Lsp$ and on $\Hsp$, we define the real inner products
\begin{equation*}
(\vbf, \wbf)_{\Lsp} 
    = \sum_{j=1}^p (v_j, w_j)_{L^2} 
\qquad \text{and} \qquad 
(\vbf, \wbf)_{\Hsp} 
    = \sum_{j=1}^p (v_j, w_j)_{H_0^1},
\end{equation*}
which induce the norms $\|\vbf\|_{\Lsp} = \sqrt{(\vbf, \vbf)_{\Lsp}}$ and $\|\vbf\|_{\Hsp} = \sqrt{(\vbf, \vbf)_{\Hsp}}$, respectively. 

For $\vbf,\wbf\in\Lsp$,  we further introduce the diagonal matrices 
\[
    \out{\vbf}{\wbf}_\C
    =\diag\big((v_1, w_1)_{L^2_{\C}},\ldots,(v_p, w_p)_{L^2_{\C}}\big) \in \DpC 
\]
 and $\out{\vbf}{\wbf}\! =\! \re \out{\vbf}{\wbf}_\C \in \DpR$.
Then the (real) inner product on $\Lsp$ reads $(\vbf, \wbf)_{\Lsp} = \trace\, \out{\vbf}{\wbf}$. For $\vbf^\star\in\Hsp^\star$ and $\wbf\in\Hsp$, we also define the diagonal matrix 
\[
\coout{\vbf^\star}{\wbf}
    =\diag\big(\langle v_1^\star, w_1\rangle_{H^{-1}\times H_0^1},\ldots,\langle v_p^\star, w_p\rangle_{H^{-1}\times H_0^1}\big) \in \DpR, 
\]
where $\langle\cdot,\cdot\rangle_{H^{-1}\times H_0^1}$ denotes the real duality pairing on the space $H^{-1}(\calD,\C)\times H^1_0(\calD,\C)$. Moreover, the real duality pairing on $\Hsp^\star\times\Hsp$ is defined as $\langle \vbf^\star,\wbf\rangle_{\Hsp^\star\times \Hsp}=\trace\coout{\vbf^\star}{\wbf}$. For brevity, the duality pairing for $p$-frames and their components will be denoted by $\langle\cdot,\cdot\rangle$.

Consider the energy functional $\calE\colon \Hsp \rightarrow \R$ defined in~\eqref{eq:energy}. 
Our goal is to compute a~{\em ground state}, which is a~global minimizer of $\,\calE$ subject to the mass constraints~\eqref{eq:mass}. By introducing the diagonal matrix $N = \diag(N_1, \ldots, N_p)$, we define the~admissible set 
\[
	\OB = \big\{ \varphibf \in \Hsp \,:\, \out{\varphibf}{\varphibf} = N \big\},
\]
which forms the \emph{infinite-dimensional complex-valued generalized oblique manifold}. Then the constrained energy minimization problem can be formulated as
\begin{equation}\label{eq:min}
	\min_{\varphibf\in\OB} \calE(\varphibf).
\end{equation} 

\begin{remark}[Phase invariance of $\calE$ and non-uniqueness of ground state]\label{rm:invariance}
The ener\-gy functional $\,\calE$ in \eqref{eq:energy} is phase invariant in the sense that $\calE(\varphibf \, \Theta) = \calE(\varphibf)$ for all $\varphibf \in \OB$ and $\Theta\in\DpS$. This immediately implies the non-uniqueness of the ground state, which we will discuss further in Section~\ref{sec:geometry}. 
\end{remark}

%
%
\subsection{Existence of a ground state}\label{sec:setting:groundstate}
We make the following assumptions on the model parameters $\V_j$, $\Omega_j$, and $\kappa_{ij}$: 
\vspace{0.4em}
\begin{itemize}[itemsep=0.2em]
\item[\bf{A1:}] For $j=1,\ldots,p$, the potentials satisfy $\V_j \in L^\infty(\calD)$ with $\V_j(x) \geq 0$ for almost all $x \in \mathcal{D}$.
\item[\bf{A2:}] For $j=1, \ldots, p$, there exists $\varepsilon_j >0$ such that 
\[
V_j(x)-\frac{1+\varepsilon_j}{4} \Omega^2_j (x_1^2 + x_2^2) \geq 0 \qquad \text{for almost all } \, x \in \mathcal{D}.
\]
\item[\bf{A3:}] The interaction matrix $K= [\kappa_{ij}]_{i,j=1}^{p,p}$ is symmetric and component-wise non-negative.
\end{itemize}
\vspace{0.4em}
These assumptions will enable us to prove the existence of a~ground state of the rotating multicomponent BEC, see Theorem~\ref{thm:existence} below. For this purpose, we reformulate the energy functional~$\,\calE$ to conveniently account for its boundedness from below. Let us first introduce a~modified potential $\VR=[\VR_1,\ldots,\VR_p]^T$ with $\VR_j= V_j - \tfrac{1}{4}\, \Omega^2_j\, \|\text{R}^T\|^2$, which is non-negative by Assumption~{\bf A2}, and the covariant gradient $\nablaR \vbf=[\nablaR v_1, \ldots, \nablaR v_p]$ with
$\nablaR v_j = \nabla v_j + \ci \tfrac{\Omega_j}{2} \text{R}^T v_j$, where
\begin{align*}
\text{R}(x) = \begin{cases}
(x_2,-x_1) & \mbox{for } d=2, \\
(x_2,-x_1,0) & \mbox{for } d=3.
\end{cases}
\end{align*}
Then Assumptions~{\bf A1} and~{\bf A2} allow us to equip $\Hsp$ with the $(\V,\Omega)$-dependent inner product 
\begin{equation}\label{R-inner-product}
(\vbf,\wbf)_{\sR} 
  = \sum_{j=1}^p \re \Big(  \int_{\calD} (\nablaR v_j)^T  \overline{\nablaR w_j}  
    +  \VR_j(x) v_j \overline{w_j} \,\dx  \Big)
\end{equation}
and the induced norm $\|\vbf\|_{\sR} = \sqrt{(\vbf, \vbf)_{\sR}}$, which is equivalent to the canonical norm $\|\cdot\|_{\Hsp}$ in the sense that there exist constants $c_1,c_2>0$ such that $c_1\|\vbf\|_{\Hsp}\leq\|\vbf\|_{\sR}\leq c_2 \|\vbf\|_{\Hsp}$. The norm equivalence immediately follows from \cite[Lem.~2.2]{DoeH23}.

With the inner product \eqref{R-inner-product}, the energy functional in \eqref{eq:energy} can be rewritten as
 \begin{equation}\label{eq:energy2}
 \calE(\varphibf) 
    = \frac {1}{2}\|\varphibf\|^2_{\sR} 
        + \frac{1}{4}\, \int_{\calD} (\varphibf\circ\overline{\varphibf})K(\varphibf\circ\overline{\varphibf})^T \, \dx,
 \end{equation}
where $\varphibf \circ \vbf = (\varphi_1 v_1, \ldots, \varphi_p v_p)$ is the Hadamard (component-\-wise) product of two $p$-frames. 
This shows that under Assumption~{\bf A3}, $\calE(\varphibf)\geq 0$ for all $\varphibf\in\Hsp$. Combined with the weak lower semi-continuity of $\,\calE$, this property is essential for establishing the existence of ground states, which can be proved analogously to \cite[Th.~2.4]{AHPS25} by adapting its proof to the inner product \eqref{R-inner-product}. 

\begin{theorem}[Existence of a~ground state]
\label{thm:existence}
Let Assumptions~{\bf A1}--{\bf A3} be fulfilled. Then there exists a~ground state $\varphibf_* \in \OB$ which is a~global minimizer of the constrained minimization problem \eqref{eq:min}. 
\end{theorem}
%

\section{Properties of the derivatives of the energy functional and nonlinear eigenvector problem}
\label{sec:derivatives}
A detailed understanding of the derivatives of the energy functional $\,\calE$ is essential for both theoretical analysis and algorithmic development. In this section, we examine the structure and properties of the first- and second-order derivatives of $\,\calE$ and characterize its constrained critical points as solutions to a~nonlinear eigenvector problem, known as the coupled Gross-Pitaevskii eigenvalue problem. 
%
%
\subsection{The first-order derivative}

To proceed, we first observe that the energy functional $\,\calE$ in~\eqref{eq:energy} is twice $\mathbb{R}$-Fr\'echet differentiable on~$\Hsp$. Using the symmetry of the interaction matrix~$K$, the first-order directional derivative of~$\,\calE$ at~$\varphibf \in \Hsp$ along $\wbf \in \Hsp$ takes the form 
\[
\Drm \calE(\varphibf)\wbf = a_{\varphibf}(\varphibf,\wbf),
\]
where the bilinear form $a_{\varphibf}:H \times H \rightarrow \mathbb{R}$ is defined by 
\begin{align}
a_{\varphibf}(\vbf,\wbf)
    &= \sum_{j=1}^p \re \int_\mathcal{D} (\nabla v_j)^T \, \overline{\nabla w_j}  
        + \V_j(x)\, v_j \overline{w_j} 
        - \Omega_j \overline{w_j}\, \spot v_j 
        + \rho_j(\varphibf)\, v_j \overline{w_j} \, \dx \label{eq:aphi} \\
   & = (\vbf,\wbf)_{\sR}+  \re \int_{\calD}\big(\varphibf \circ \overline{\varphibf}\big)K\big(\vbf \circ \overline{\wbf}\big)^T \, \dx. \label{eq:aphi_short}
\end{align}
The next proposition presents the fundamental properties of this bilinear form and can be verified by direct calculation. 
\begin{proposition}\label{prop:aphi}
Let Assumptions~{\bf A1}--{\bf A3} be fulfilled. Then for any $\varphibf\in \Hsp$,  the bilinear form~$a_\varphibf$ defined in \eqref{eq:aphi} is symmetric, bounded, and coercive with the coercivity constant of~$\,1$ with respect to the norm $\|\cdot\|_{\sR}$.  
\end{proposition}
For any fixed $\varphibf \in \Hsp$, the coercive bilinear form $a_\varphibf$ defines an~alternative inner product on~$\Hsp$ and induces the norm $\|\vbf\|_{a_\varphibf} = \sqrt{a_\varphibf(\vbf, \vbf)}$, which is equivalent to $\|\cdot\|_{\sR}$. 
Moreover, the bilinear form $a_\varphibf$ defines the \emph{Gross-Pitaevskii Hamiltonian} $\calA_\varphibf\colon \Hsp\to \Hsp^\star$ given by
\[
\langle \calA_\varphibf \vbf,\wbf\rangle = a_\varphibf(\vbf,\wbf)  
\qquad \text{for all }\vbf,\wbf\in \Hsp.
\]
Due to the additive structure of $a_\varphibf$ in \eqref{eq:aphi}, the operator $\calA_\varphibf$ can be represented as 
\begin{equation}\label{eq:calA}
\langle \calA_\varphibf \vbf,\wbf\rangle 
= \sum_{j=1}^p \langle \calA_{\varphibf,j} v_j,w_j\rangle
\qquad \text{for all }\vbf,\wbf\in \Hsp,
\end{equation}
where the component operators $\calA_{\varphibf,j}\colon H_0^1(\DC)\to H^{-1}(\DC)$ are given by
\begin{equation}\label{eq:calAj}
    \langle\calA_{\varphibf,j}v_j,w_j\rangle 
    = \re \int_\mathcal{D} (\nabla v_j)^T \, \overline{\nabla w_j}  
	+ \V_j(x)\, v_j \, \overline{w_j} - \Omega_j \overline{w_j} \spot v_j 
    + \rho_j(\varphibf)\, v_j \, \overline{w_j}\, \dx, \quad j=1,\ldots,p.
\end{equation}
It follows from \eqref{eq:calA} that the operator $\calA_\varphibf$ acts component-wise on a $p$-frame, meaning that 
\[
\calA_\varphibf\,\vbf = (\calA_{\varphibf,1}v_1, \ldots,\calA_{\varphibf,p}v_p).
\]
This property immediately implies that $\calA_\varphibf$ is right-equivariant on $\DpS$ in the sense that for all $\vbf \in \Hsp$ and $\Theta \in \DpS$, we have $\calA_\varphibf(\vbf\Theta) = (\calA_\varphibf\vbf)\Theta$. Furthermore, for all $\Theta\in\DpS$, it holds that $\calA_{\varphibf\Theta}=\calA_{\varphibf}$, which means that $\calA_{\varphibf}$ is phase invariant.

\begin{proposition}
The mapping $\varphibf\mapsto a_\varphibf$ with the bilinear form $a_\varphibf$ defined in \eqref{eq:aphi_short} is conti\-nuous as a~function from $[L^4(\DC)]^p$ to the space of bilinear forms on $\Hsp$ in the sense that for all $\varphibf, \psibf \in [L^4(\DC)]^p$ and $\vbf$, $\wbf \in \Hsp$, there exists a constant $C_{\varphibf,\psibf} > 0$ such that
\begin{equation}\label{eq:boundedness}
|a_\varphibf(\vbf,\wbf) - a_{\psibf}(\vbf,\wbf)|\leq C_{\varphibf,\psibf}\|\vbf\|_H\|\wbf\|_H\|\varphibf-\psibf\|_{[L^4(\calD,\C)]^p}.
\end{equation}
\end{proposition}
\begin{proof}
For all  $\varphibf, \psibf \in [L^4(\DC)]^p$ and $\vbf,\wbf\in\Hsp$, we obtain from \eqref{eq:aphi_short} that
\begin{align*}
|a_\varphibf(\vbf,\wbf) - a_{\psibf}(\vbf,\wbf)| 
& = \bigg|\int_\calD(\varphibf \circ \overline{\varphibf} - \psibf \circ \overline{\psibf})K \big(\re(\vbf \circ \overline{\wbf}) \big)^T \dx\bigg| \\
 & \leq C_4^2\|K\|_2 \big\||\varphibf| + |\psibf|\big\|_{[L^4(\DC)]^p}\big\|\varphibf - \psibf\big\|_{[L^4(\DC)]^p}\big\|\vbf\big\|_\Hsp\big\|\wbf\big\|_\Hsp,
\end{align*}
where the constant $C_4$ arises from the Sobolev embedding $\Hsp\hookrightarrow [L^4(\DC)]^p$.
This implies \eqref{eq:boundedness} with $C_{\varphibf,\psibf}=C_4^2\|K\|_2 \big\||\varphibf| + |\psibf|\big\|_{[L^4(\DC)]^p}$.
\end{proof}

Note that the coercivity of the bilinear form $a_{\varphibf}$ ensures the existence of the inverse operator $\calA_\varphibf^{-1}\colon \Hsp^\star\to \Hsp$ defined as $a_\varphibf(\calA_\varphibf^{-1}\vbf^{\star},\wbf)=\langle\vbf^{\star},\wbf\rangle$ for all $\vbf^{\star}\in\Hsp^\star$ and $\wbf\in\Hsp$. Due to the canonical inclusion~$\Hsp \subset \Hsp^\star$, the application $\calA_\varphibf^{-1}\vbf$ to functions $\vbf \in \Hsp$ is well-defined. Obviously,~$\calA_\varphibf^{-1}$ acts component-wise. The following proposition shows that it also inherits further properties of~$\calA_\varphibf$ such as phase invariance and right-equivariance on $\DpS$. 
\begin{proposition}\label{prop:inv_A_property}
    For all $\varphibf,\vbf \in \Hsp$ and $\Theta \in \D(p,\SSS)$, we have 
\begin{equation}\label{eq:invAprop}
    \calA^{-1}_{\varphibf \Theta} \vbf = \calA^{-1}_{\varphibf} \vbf, \qquad 
    \calA^{-1}_{\varphibf} (\vbf \Theta) = (\calA^{-1}_{\varphibf} \vbf) \Theta.
\end{equation}
\end{proposition}
\begin{proof}
    For all $\varphibf,\vbf,\wbf \in \Hsp$ and $\Theta \in \D(p,\SSS)$, we obtain that 
    \begin{align*}
    a_{\varphibf}\big(\calA_{\varphibf}^{-1} \vbf, \wbf\big)
        & = (\vbf, \wbf)_{\Lsp} 
          = a_{\varphibf \Theta}\big(\calA^{-1}_{\varphibf \Theta} \vbf,\wbf  \big) 
          = a_{\varphibf }\big(\calA^{-1}_{\varphibf \Theta} \vbf,\wbf\big), \\
    a_{\varphibf}\big(\calA_{\varphibf}^{-1} \vbf, \wbf\big)
        & = (\vbf \Theta , \wbf \Theta)_{\Lsp} 
        = a_{\varphibf}\big(\calA^{-1}_{\varphibf} (\vbf \Theta),\wbf \Theta \big) 
        = a_{\varphibf }\big(\calA^{-1}_{\varphibf} (\vbf \Theta)\overline{\Theta},\wbf \big).
    \end{align*}
    These relations immediately imply \eqref{eq:invAprop}.
\end{proof}

%
%
\subsection{The second-order derivative}
The second-order directional derivative of $\,\calE$ at $\varphibf\in\Hsp$ in the direction of $\vbf, \wbf \in \Hsp$ is computed~as 
\begin{align}\label{eq:HessE}
	\langle \Drm^2 \calE(\varphibf)\vbf,\wbf\rangle
	& = \lim_{t\to 0} \tfrac{1}{t}\, \big\langle \calA_{\varphibf+t\vbf}(\varphibf+t\vbf)-\calA_\varphibf\, \varphibf ,\wbf  \big\rangle 
 = \big\langle \calA_\varphibf \,\vbf +\calB_{\varphibf}(\vbf,\varphibf), \wbf\big\rangle,
\end{align}
where the operator $\calB_\varphibf\colon \Hsp\times \Hsp \to \Hsp^\star$ is given by 
\[
\langle\calB_{\varphibf}(\vbf,\ubf), \wbf\rangle 
    = 2 \int_\mathcal{D} \big(\re(\varphibf \circ \overline{\vbf})\big) K \big(\re (\ubf \circ \overline{\wbf})\big)^T \, \mathrm{d}x, \qquad \ubf, \vbf, \wbf \in \Hsp. 
\]
This operator can also be written in additive form
\[
\langle\calB_{\varphibf}(\vbf,\ubf), \wbf\rangle 
 = \sum_{i,j=1}^p \langle\calB_{\varphibf,ij}(v_j,u_i), w_i\rangle
 \]
 with the operators 
 \begin{equation}\label{eq:Bij}
 \langle\calB_{\varphibf,ij}(v_j,u_i), w_i\rangle
 = 2 \,\kappa_{ij}\int_\mathcal{D}  \re(\varphi_j \overline{v_j}) \, \re (u_i \overline{w_i}) \dx, \qquad i,j=1,\ldots, p. 
\end{equation}
Key properties of the $\R$-Fr\'echet derivative $\Drm^2\calE(\varphibf)$ are collected in the following proposition.

\begin{proposition}\label{prop:D2Eproperties}
Let Assumptions~{\bf A1}--{\bf A3} be fulfilled. Then for any $\varphibf\in \Hsp$, the second-order derivative~$\Drm^2\calE(\varphibf)$  defined in \eqref{eq:HessE} is symmetric and bounded. If, in addition, $K$ is positive definite, then $\Drm^2\calE(\varphibf)$ is coercive on $\Hsp$.
\end{proposition}
\begin{proof} 
The first result follows from Proposition~\ref{prop:aphi} and the symmetry and boundedness of the linear operator $\calB_\varphibf(\,\cdot\,,\varphibf)$. If $K$ is positive definite, then $\langle\calB_\varphibf(\vbf,\varphibf),\vbf\rangle\geq 0$ for all $\vbf\in H$. Hence, using Proposition~\ref{prop:aphi} once more, we find that $\Drm^2\calE(\varphibf)$ is coercive on $\Hsp$.
\end{proof}

%
%
\subsection{Nonlinear eigenvector problem}
\label{ssec:first_order_optimality}
A~relation between the constrained energy minimization problem~\eqref{eq:min} and a~nonlinear eigenvector problem can be established by introducing the Lagrangian 
\[
	\calL(\varphibf,\Lambda) 
	= \calE(\varphibf) - \frac{1}{2} \trace\Big(\Lambda \big(\out{\varphibf}{\varphibf}-N\big) \Big)
\]
with a~Lagrange multiplier $\Lambda=\diag(\lambda_1,\ldots,\lambda_p)\in\DpR$. The directional derivative of $\calL$ with respect to $\varphibf\in\Hsp$ along $\wbf\in\Hsp$ has the form  
\begin{equation}\label{eq:derL}
	\Drm_\varphibf \calL(\varphibf,\Lambda)\wbf
	 = \Drm\calE(\varphibf)\wbf-(\varphibf\,\Lambda,\wbf)_\Lsp = \langle\calA_\varphibf\,\varphibf,\wbf\rangle - (\varphibf\,\Lambda,\wbf)_\Lsp. 
\end{equation}
A $p$-frame $\varphibf \in \Hsp$ is called a \emph{constrained critical point} of the energy functional~$\,\calE$ if the first-order optimality conditions are satisfied, which means that $\varphibf \in \OB$, and there exists a Lagrange multiplier $\Lambda \in \DpR$ such that $\Drm_\varphibf \calL(\varphibf,\Lambda)\wbf = 0$ for all $\wbf \in \Hsp$. Due to \eqref{eq:derL}, the computation of the constrained critical points of $\,\calE$ is therefore linked to the following nonlinear eigenvector problem: find $\varphibf\in\OB$ and $\Lambda\in\DpR$ such that 
\begin{equation}\label{eq:nlevp}
 \langle\calA_\varphibf\,\varphibf, \wbf\rangle 
    = (\varphibf\,\Lambda,\wbf)_\Lsp \qquad \text{for all } \wbf \in \Hsp.
\end{equation}
Given a constrained critical point $\varphibf \in \OB$, the corresponding Lagrange multiplier~$\Lambda$ can then be determined as $\Lambda = \coout{\calA_\varphibf\,\varphibf}{\varphibf}N^{-1}$.

Using the additive representation \eqref{eq:calA}, we rewrite \eqref{eq:nlevp} component-wise as
\begin{equation}\label{eq:levp1}
    \langle \calA_{\varphibf,j} \varphi_j, w_j \rangle
    = \lambda_j (\varphi_j, w_j)_{L^2} \qquad \text{for all } w_j \in H_0^1(\DC).
\end{equation}
Then the expression for $\calA_{\varphibf,j}$ in \eqref{eq:calAj} yields the coupled Gross-Pitaevskii eigenvalue problem
\[
	-\Delta \varphi_j + \V_j \varphi_j - \Omega_j  \spot \varphi_j + \rho_j(\varphibf)\varphi_j 
    = \lambda_j\varphi_j, \qquad j=1,\ldots,p.
\]
Given a pair consisting of a $p$-frame $\varphibf$ and a~diagonal matrix $\Lambda = \diag(\lambda_1, \ldots, \lambda_p)$, that satisfies~\eqref{eq:nlevp} or \eqref{eq:levp1}, we call $\varphibf$ an~\emph{eigenvector} corresponding to the \emph{eigenvalues} $\lambda_1, \ldots, \lambda_p$.

Note that due to the first-order optimality condition, a~ground state \mbox{$\varphibf_*\in\OB$}, being the global minimizer of $\,\calE$, is a~constrained critical point. In the non-rotating case ($\Omega=0$), it has been shown in \cite[Prop.~4]{AHPS25} that for a~ground state $\varphibf_*$ and the corresponding Lagrange multiplier 
\mbox{$\Lambda_* = \diag(\lambda_{*,1}\ldots,\lambda_{*,p})$}, the smallest eigenvalue of the component operator $\calA_{\varphibf_*,j} $ coincides with~$\lambda_{*,j}$ for $j=1,\ldots,p$. However, for rotating multicomponent BECs, this component-wise minimality property no longer holds in general, which makes the convergence analysis of iterative methods for computing the ground state more challenging; see \cite{HenY25} for the related discussion in the single-component case. There, it has been proposed to work with the spectrum of the second-order $\R$-Fr\'echet derivative $\Drm^2\calE(\varphibf_*)$ instead. To extend these results to the multicomponent case and to address the phase-induced non-uniqueness of the ground state, we need to analyze the quotient nature of the mass constraints \eqref{eq:mass}.

%
\section{Quotient geometry}\label{sec:geometry}
As mentioned in Remark~\ref{rm:invariance}, the energy functional~$\,\calE$ in~\eqref{eq:energy} is invariant with respect to multiplications by an~arbitrary phase matrix $\Theta \in \DpS$. When encoding the mass constraints \eqref{eq:mass} geometrically, it is natural to identify states that differ only in phase. This leads to the concept of quotient manifolds. In the following, we introduce such a~manifold first for the case of a single component and then extend it to the multicomponent setting. 

In the single-component case ($p=1$), the generalized oblique manifold $\mathcal{O}\mathcal{B}_N^{\,\C}(1,\Hsp)$ reduces to the \emph{complex-valued sphere} 
\[
\Sphere_{N_1} = \{\varphi \in H^1_0(\mathcal{D},\C)\enskip:\enskip(\varphi,\varphi)_{L^2} = N_1\}.
\]
Its tangent space at $\varphi\in\Sphere_{N_1}$ is given by
\[
T_{\varphi}\,\Sphere_{N_1} = \{ z \in H^1_0(\mathcal{D},\C) \enskip : \enskip (z,\varphi)_{L^2} = 0 \}.
\]
Endowing this space with a~Riemannian metric $g_\varphi: T_\varphi\,\Sphere_{N_1} \times T_\varphi\,\Sphere_{N_1} \to \mathbb{R}$ turns $\Sphere_{N_1}$ into a~Riemannian Hilbert manifold. 

Using the equivalence relation on $\Sphere_{N_1}$ defined as $\varphi \sim \psi$  if and only if $\varphi = \psi\,\theta$ for some $\theta \in \SSS$, we can form the quotient space $\Sphere_{N_1}/\SSS$ as the set of all equivalence classes $[\varphi]=\{\psi\in\Sphere_{N_1}\;:\; \varphi\sim\psi\}$ for $\varphi\in\Sphere_{N_1}$. To make a link between the sphere $\Sphere_{N_1}$ and the quotient space $\Sphere_{N_1}/\SSS$, we define the canonical projection $\pi_1:\Sphere_{N_1}\to\Sphere_{N_1}/\SSS$ as $\pi_1(\varphi)=[\varphi]$. Observe that $\mathbb{S}$ is a~compact Lie group and the map $(e^{\ci \alpha},v) \mapsto v e^{\ci \alpha}$ is a smooth, proper and free action. Hence, $\pi_1$ is a~smooth submersion and  $\Sphere_{N_1}/\SSS$ admits a~quotient manifold structure \cite[Prop.~5.3.2]{AbrMR88}. In order to characterize the tangent space to $\Sphere_{N_1}/\SSS$, we decompose the tangent space to $\Sphere_{N_1}$ as 
$T_\varphi\,\Sphere_{N_1}=\calV_{\varphi}\oplus\calH_{\varphi}^g$, where 
$\calV_{\varphi} = \{ \ci\varphi\sigma \enskip:\enskip \sigma\in \R\}$ is the \textit{vertical space}, which does not depend on the metric, and 
\[
\Horg{\varphi}{g} = \{\xi \in T_\varphi\Sphere_{N_1} \enskip:\enskip g_\varphi(\nu,\xi) = 0 \text{ for all } \,\nu \in \mathcal{V}_\varphi\}
\]
is the $g_\varphi$-orthogonal \textit{horizontal space}.
The vertical space $\calV_\varphi$ captures the directions in which the equivalence class does not change when moving on a~curve in $\Sphere_{N_1}$. 
Since $\Drm\!\pi_1(\varphi)$ is bijective on the horizontal space $\calH_\varphi^g$, the tangent space $T_{[\varphi]}\,\Sphere_{N_1}/\SSS$ to $\Sphere_{N_1}/\SSS$ at $[\varphi] \in \Sphere_{N_1}/\SSS$ can be identified with~$\calH_{\varphi}^g$ in the sense that for all $\widetilde{z}\in T_{[\varphi]}\,\Sphere_{N_1}/\SSS$, there exists a~unique \textit{horizontal lift} $z_\varphi^{\hrm,g} \in \Horg{\varphi}{g}$ such that $\Drm\!\pi_1(\varphi)z^{\hrm,g}_\phi=\widetilde{z}$.
We refer to \cite{Boumal23} for further details for quotient manifolds in the matrix case and to \cite{AltPS24,BilM14} for related discussions for the Stiefel manifold and its corresponding quotient manifold, the Grassmann manifold, in the infinite-dimensional setting.

To apply these ideas to the multicomponent case, we notice that the generalized oblique manifold $\OB$ is a~product manifold of $p$ spheres 
\[
    \OB = \Sphere_{N_1} \times \ldots \times \Sphere_{N_p},
\]
and the tangent space to this manifold at $\varphibf = (\varphi_1, \ldots, \varphi_p) \in \OB$ is given by 
\[
T_\varphibf\, \OB = \big\{\zbf \in \Hsp \enskip : \enskip\out{\varphibf}{\zbf}  = 0_p\big\}
=T_{\varphi_1}\Sphere_{N_1} \times \ldots \times T_{\varphi_p}\Sphere_{N_p}.
\]
This naturally leads to the product manifold $\calM^\times=\Sphere_{N_1}/\SSS \times \ldots \times \Sphere_{N_p}/\SSS$ with the usual product smooth structure. 

Alternatively, we can define the quotient space $\Seg = \OB/\DpS$ as the set of all equivalence classes $[\varphibf]=\{\psibf\in\OB\; : \; \varphibf=\psibf\Theta, \; \Theta\in\DpS\}$. Since $\DpS$ is a~compact Lie group acting smoothly, freely, and properly on $\OB$ from the right, the quotient space~$\Seg$ is a quotient manifold of co-dimension~$p$, and the canonical projection $\pi: \OB \rightarrow \Seg$ defined as $\pi(\varphibf) = [\varphibf]$ is a smooth submersion \cite[Prop.~5.3.2]{AbrMR88}. Using \cite[Th.~4.31]{Lee13}, the quotient manifold~$\Seg$ can be identified with the product manifold $\calM^{\times}$ and, as a consequence, the above results can be extended to the multicomponent setting. 

Given a Riemannian metric $g_\varphibf : T_{\varphibf}\,\OB \times T_{\varphibf}\,\OB \to \R$, the tangent space $T_\varphibf\, \OB $ can be decomposed as $T_\varphibf\, \OB=\calV_\varphibf\oplus\calH_\varphibf^g$ into the vertical space 
\begin{equation}\label{eq:vert_space}
\calV_{\varphibf} = \big\{\ci\varphibf\,\Sigma \enskip : \enskip \Sigma \in \DpR \big\} 
=\calV_{\varphi_1}\times\ldots\times \calV_{\varphi_p}
\end{equation}
and its $g_\varphibf$-orthogonal complement, the horizontal space
\begin{align*}
    \Horg{\varphibf}{g} = \{\xibf \in T_{\varphibf}\,\OB \enskip : \enskip g_\varphibf(\xibf,\nubf) = 0 \text{ for all } \,\nubf \in \mathcal{V}_\varphibf\}.
\end{align*}
As above, the significance of $\Horg{\varphibf}{g}$ is that it provides a~representation for any tangent vector \mbox{$\widetilde\zbf \in T_{[\varphibf]}\Seg$} via a~unique horizontal lift $\widetilde\zbf_\varphibf^{\hrm,g} \in \Horg{\varphibf}{g}$ such that $\Drm\!\pi(\varphibf)\widetilde\zbf_\varphibf^{\hrm,g}=\widetilde\zbf$. Note that exploiting the product structure of $\OB$ and defining a~\emph{product metric} as a~sum of metrics on the spheres $\calS_{N_j}$, i.e.,
\begin{equation}\label{eq:prod_metric}
g_\varphibf(\ybf,\zbf)=g_{\varphi_1,1}(y_1,z_1)+\ldots+g_{\varphi_p,p}(y_p,z_p)
\qquad \text{for all }\ybf,\zbf\in T_{\varphibf}\,\OB,
\end{equation}
the horizontal space $\calH_{\varphibf}^g$ can be expressed as a product $\calH_{\varphibf}^g=\calH_{\varphi_1}^{g_1}\times\ldots\times\calH_{\varphi_p}^{g_p}$ of the horizontal spaces $\calH_{\varphi_j}^{g_j}$ to the spheres $\calS_{N_j}$ endowed with the metrics $g_{\varphi_j,j}$, respectively.
If, in addition, we assume that the product metric is \emph{phase invariant} in the sense that 
\begin{equation}\label{eq:rotinvar}
    g_{\varphibf\Theta}(\ybf\Theta,\zbf\Theta) = g_\varphibf(\ybf,\zbf)\qquad \text{for all }
    \ybf,\zbf\in T_\varphibf\,\OB, \;\Theta\in\DpS, 
\end{equation} 
then we can prove that the horizontal lifts of the tangent vectors at points from the same equivalence class are related by a~scaling matrix from $\DpS$.

\begin{proposition}\label{prop:phase_lift}
Let $\varphibf\in\OB$. Assume that $\OB$ is endowed with a~phase invariant product metric $g_\varphibf$. Then for all $\widetilde\zbf \in T_{[\varphibf]}\Seg$ and all $\Theta\in\DpS$, the horizontal lifts of $\widetilde\zbf$ at $\varphibf$ and $\varphibf\Theta$ are related by $\widetilde\zbf^{\hrm,g}_{\varphibf\Theta} = \widetilde\zbf^{\hrm,g}_{\varphibf}\Theta$.
\end{proposition}
\begin{proof}
    The result can be shown by adapting the proof in \cite[Ex.~9.26]{Boumal23} to the generalized oblique manifold. Let $\varphibf\in\OB$ and let $\widetilde\zbf^{\hrm,g}_{\varphibf}$ be a horizontal lift of $\widetilde{\zbf}$ at $\varphibf$. For any smooth curve $\gamma:(-1,1)\to\OB$ with $\gamma(0)=\varphibf$ and $\gamma'(0)=\widetilde\zbf^{\hrm,g}_{\varphibf}$, we have $\pi(\gamma(0))=[\varphibf]$ and $\frac{\rm d}{{\rm d}t}\pi(\gamma(t))\big|_{t=0}=\Drm\!\pi(\gamma(0))\gamma'(0)=\widetilde{\zbf}$. For any $\Theta\in\DpS$,  we define another curve $\gamma_\Theta(t)=\gamma(t)\Theta$, $t\in(-1,1)$. Then $\gamma_\Theta(0)=\varphibf\Theta$, $\gamma'_\Theta(0)=\widetilde\zbf^{\hrm,g}_{\varphibf}\Theta$, and $\pi(\gamma_\Theta(t))=\pi(\gamma(t))$. Differentiating the latter, we get 
    \[
    \widetilde{\zbf}
    =\dfrac{\rm d}{{\rm d}t}\pi(\gamma(t))\Big|_{t=0}
    =\dfrac{\rm d}{{\rm d}t}\pi(\gamma_\Theta(t))\Big|_{t=0}
    =\Drm\!\pi(\gamma_\Theta(0))\gamma'_\Theta(0)
    =\Drm\!\pi(\varphibf\Theta)(\widetilde\zbf^{\hrm,g}_{\varphibf}\Theta).
    \]
    Due to the phase invariance of the metric, we have $\widetilde\zbf^{\hrm,g}_{\varphibf}\Theta \in\calH_{\varphibf\Theta}^g$. Then the uniqueness of horizontal lifts finally implies that $\widetilde\zbf^{\hrm,g}_{\varphibf}\Theta$ is the horizontal lift of $\widetilde{\zbf}$ at $\varphibf\Theta$, i.e., 
    $\widetilde\zbf^{\hrm,g}_{\varphibf\Theta}=\widetilde\zbf^{\hrm,g}_{\varphibf}\Theta$.
\end{proof}

For the $\Lsp$-metric defined as $g_{\Lsp}(\ybf,\zbf) = (\ybf,\zbf)_\Lsp$ for all $\ybf,\zbf\in T_\varphibf\OB$, the associated horizontal space $\Horg{\varphibf}{\Lsp}$ at $\varphibf \in \OB$ has a~particularly simple structure, which follows directly from its definition: 
\begin{equation}\label{eq:hor_tan}
    \Horg{\varphibf}{\Lsp} = \big\{\xibf \in \Hsp \enskip : \enskip  \out{\varphibf}{\xibf}_\C = 0_p\big\} 
    = T_{\varphibf}\OB \cap T_{\ci\varphibf}\OB.
\end{equation}
This connects our results below to those obtained in \cite{HenY25}, where the right-hand side of \eqref{eq:hor_tan} is used extensively. 
Although we will not directly work with the $\Lsp$-metric in the optimization methods, since the tangential space $T_{\varphibf}\OB$ is not complete with respect to this metric, we will see below that at a~ground state~$\varphibf_*$, relation \eqref{eq:hor_tan} also holds for the horizontal spaces with respect to other relevant metrics. Motivated by this observation, it is natural to introduce a notion of compatibility between the Riemannian metrics~$g_{\varphibf}$ and~$g_{\Lsp}$ at $\varphibf \in \OB$.

\begin{definition} \label{def:hor_compat}
A Riemannian metric $g_{\varphibf}$ is called \emph{horizontally compatible with the \mbox{$\Lsp$-metric} $g_{\Lsp}$} at $\varphibf \in \OB$, if the associated horizontal spaces coincide, i.e., \mbox{$\Horg{\varphibf}{g} = \Horg{\varphibf}{\Lsp}$}.
\end{definition}

If $g_\varphibf$ is defined on the whole space $\Hsp$, it induces a linear operator $\calG_\varphibf: \Hsp \to \Hsp^{\star}$ given by $g_\varphibf(\vbf,\wbf)\! = \!\langle\calG_\varphibf\,\vbf,\wbf\rangle$. For a~pro\-duct metric as in \eqref{eq:prod_metric}, this operator acts component-wise, i.e., $\calG_\varphibf\,\vbf = (\calG_{\varphibf,1}v_1, \dots, \calG_{\varphibf,p}v_p)$. The following lemma gives a characterization of horizontal compatibility in terms of $\calG_{\varphibf,j}$.

\begin{proposition}\label{prop:hor_compat_eigvalue}
Let $\varphibf \in \OB$ and $g_\varphibf$ be a product metric with the associated component operators $\calG_{\varphibf,j}: H_0^1(\calD,\C) \to H^{-1}(\calD,\C)$. Then $g_\varphibf$ is horizontally compatible with the $\Lsp$-metric $g_{\Lsp}$ at $\varphibf$ if and only if for every $j = 1, \dots, p$, $\ci\varphi_j$ is an eigenfunction of $\calG_{\varphibf,j}$ on $T_{\varphi_j}\Sphere_{N_j}$ to a~nonzero real eigenvalue.
\end{proposition}
\begin{proof}
We start by noticing that $\Horg{\varphibf}{\Lsp}$ consists of those $\xibf \in T_\varphibf\OB$ that satisfy \mbox{$\out{\ci\varphibf}{\xibf}\! = 0$}, while $\xibf \in \Horg{\varphibf}{g}$ is equivalent to $\xibf \in T_\varphibf\OB$ and $\coout{\calG_\varphibf(\ci\varphibf)}{\xibf} = 0$. This immediately implies the "if" part. 
For the other direction, let $\varsigma_j = N_j^{-1}\langle\calG_{\varphibf,j}(\ci\varphi_j),\ci\varphi_j\rangle > 0$ for $j = 1, \dots p$. For any $\zbf = (z_1, \dots, z_p) \in T_\varphibf\OB$, it holds that $z_j = N_j^{-1}(\ci\varphi_j,z_j)_{L^2}\ci\varphi_j + z_j^\perp$, where $(\ci\varphi_j,z_j^\perp)_{L^2} = \langle\calG_\varphibf(\ci\varphi_j),z_j^\perp\rangle = 0$. Then we have
$
    \langle\calG_{\varphibf,j}(\ci\varphi_j),z_j\rangle
    = \varsigma_j(\ci\varphi_j,z_j)_{L^2},
$
and therefore $\ci\varphi_j$ is an eigenfunction of $\calG_{\varphibf,j}$ on $T_{\varphi_j}\Sphere_{N_j}$ to the eigenvalue $\varsigma_j$.
\end{proof}

If $g_\varphibf$ is bounded on $\Hsp$ and coercive with respect to the $\Hsp$-norm, the linear operator~$\calG_\varphibf$ is invertible. Using this, we can write the \emph{$g_\varphibf$-orthogonal projection} onto the tangent space $T_\varphibf\OB$ at $\varphibf \in \OB$ as 
\begin{equation}\label{eq:projection}
    \calP_\varphibf^g \vbf = \vbf - \calG_\varphibf^{-1}\varphibf\out{\varphibf}{\vbf}\out{\varphibf}{\calG_\varphibf^{-1}\varphibf}^{-1}, \qquad \vbf\in\Hsp.
\end{equation}
Although the $\Lsp$-metric is not $\Hsp$-coercive, we analogously define the $\Lsp$-orthogonal projection onto 
$T_\varphibf\OB$ as $\calP_\varphibf^\Lsp \vbf = \vbf - \varphibf\out{\varphibf}{\vbf}N^{-1}$.


\section{Quasi-uniqueness of ground states within the quotient manifold framework}
\label{sec:uniqueness}
To specifically address the non-uniqueness of the ground state arising from rotational symmetry of the energy functional $\,\calE$, we make use of the quotient manifold $\Seg$, which allows us to reformulate the constrained minimization problem \eqref{eq:min}~as
\begin{equation}\label{eq:min_quotient}
    \min_{[\varphibf]\in\Seg} \widetilde{\calE}([\varphibf]),
\end{equation}
where the cost functional $\,\widetilde\calE:\Seg\to\R$ is induced by $\,\calE$ as 
$\widetilde{\calE}\big([\varphibf]\big) = \calE\big(\pi^{-1}([\varphibf])\big)$, 
$\calE(\varphibf)=\widetilde\calE\big(\pi(\varphibf)\big)$.
The existence of a~global minimizer to \eqref{eq:min} established in Theorem \ref{thm:existence} implies that \eqref{eq:min_quotient} has also a~solution, but now we can hope for local uniqueness, as we shall examine in this section.

Intuitively, movements in the vertical direction do not affect the constrained energy $\,\calE$. Mathe\-matically, this leads to the non-invertibility of the Lagrangian Hessian
$\Drm_{\varphibf\varphibf}^2\calL(\varphibf_*,\Lambda_*)$ on the vertical space~$\calV_{\varphibf_*}$ at a~ground state~$\varphibf_*$. 
Indeed, since the pair $(\ci\varphibf_*,\Lambda_*)$ 
solves the eigenvector problem~\eqref{eq:nlevp}, and $\calB_{\varphibf_*} (\nubf,\varphibf_*)=0$ for all $\nubf \in \calV_{\varphibf_*}$, 
we obtain by using \eqref{eq:vert_space} that
\begin{align}\label{eq:D2_lambda}
\big\langle\Drm_{\varphibf\varphibf}^2\calL(\varphibf_*,\Lambda_*)\nubf,\wbf\big\rangle
& = \big\langle\Drm^2 \calE(\varphibf_*)\nubf - \nubf\Lambda_*, \wbf\big\rangle 
= \big\langle\calA_{\varphibf_*} \nubf+\calB_{\varphibf_*} (\nubf,\varphibf_*)-\nubf \Lambda_* , \wbf \big\rangle = 0.
\end{align}
To resolve this issue, we introduce the concept of a~locally quasi-unique ground state, which relies on the second-order sufficient optimality condition for a~strict local minimum. 

\begin{definition}
\label{definition-quasi-isolated-ground-state}
A ground state $\varphibf_* \in \OB$ to~\eqref{eq:min} is called \emph{locally quasi-unique} if for any smooth curve $\gamma: (-1,1) \rightarrow \OB$ with $\gamma(0)=\varphibf_*$ and $\gamma^{\prime}(0) \in T_{\varphibf_*}\OB \setminus\calV_{\varphibf_*}$, it holds that
\begin{equation}\label{eq:isolated_min}
     \frac{{\rm d}^2}{{\rm d}t^2} \calE(\gamma(t)) \Big|_{t=0}  > 0
\end{equation}
or, equivalently, if $[\varphibf_*] \in \Seg$ is a strict local minimizer of \eqref{eq:min_quotient}. 
\end{definition}

Broadly speaking, the above definition states that if any two ground states of the minimization problem~\eqref{eq:min} are locally quasi-unique, then they are either well separated or related by a~complex phase shift. Equivalently, the relation~\eqref{eq:isolated_min} states that for the ground state $\varphibf_*$ and the corresponding Lagrange multiplier $\Lambda_*$, the Hessian $\Drm^2_{\!\varphibf\varphibf} \calL(\varphibf_*,\Lambda_*)$ is strictly positive on $T_{\varphibf_*}\OB \setminus \calV_{\varphibf_*}$, i.e., for all $\zbf \in T_{\varphibf_*}\OB \setminus \calV_{\varphibf_*}$, we have
\begin{equation}\label{eq:D2Lpos}
    \big\langle\Drm^2_{\!\varphibf\varphibf} \calL(\varphibf_*,\Lambda_*)\zbf,\zbf\big\rangle = \big\langle\Drm^2\calE(\varphibf_*)\zbf - \zbf \Lambda_*, \zbf\big\rangle > 0.
 \end{equation}

For $j=1,\ldots,p$, we consider the operator $\DDEjj(\varphibf_*):H^1_0(\DC)\to H^{-1}(\DC)$ given by 
\begin{equation}\label{eq:second_derivative_componentwise}
\big\langle \DDEjj(\varphibf_*)v, w\big\rangle 
    = \big\langle \calA_{\varphibf_*,j} v + \calB_{\varphibf_*,jj}(v,\varphi_{*,j}), w \big\rangle, \qquad v,w\in H_0^1(\DC),
\end{equation}
where $\calA_{\varphibf_*,j}$ and $\calB_{\varphibf_*,jj}$ are as in \eqref{eq:calAj} and \eqref{eq:Bij}, respectively.
Then for any product metric $g_\varphibf$ on $\OB$ satisfying \eqref{eq:prod_metric}, it follows from \eqref{eq:D2Lpos} that for all $\xi\in\calH_{\varphi_{*,j}}^{g_j}\setminus\{0\}$,
\[
\big\langle\Drm^2_{\!\varphi_j\varphi_j} \calL(\varphibf_*,\Lambda_*)\xi,\xi\big\rangle
= \big\langle\Drm_{\!\varphi_j\varphi_j}^2\calE(\varphibf_*)\xi, \xi\big\rangle
    - (\lambda_{*,j} \xi, \xi)_{L^2} > 0.
\]
In Proposition~\ref{prop:coercive} below, we establish that the Hessian $\Drm^2_{\!\varphi_j\varphi_j} \calL(\varphibf_*,\Lambda_*)$ is even coercive on the horizontal space $\Horg{\varphi_{*,j}}{g_j}$. For this purpose, we consider the following eigenvalue problems: for $j=1,\ldots,p$, find $v_{i,j}\in T_{\varphi_{*,j}}\Sphere_{N_j}$ and $\lambda_{i,j} \in \mathbb{R}$ such that 
\begin{equation}\label{eq:evp4D2Ejj}
 \big\langle \DDEjj(\varphibf_*)v_{i,j}, z\big\rangle
 = (\lambda_{i,j} \, v_{i,j}, z)_{L^2} \qquad \mbox{for all} \quad z \in  T_{\varphi_{*,j}}\Sphere_{N_j}.
\end{equation}
The essential properties of the eigenpairs to these problems can be characterized as follows.
\begin{proposition}[Eigenpairs of $\DDEjj(\varphibf_*)$] 
\label{prop:simple_ev} 
 Let Assumptions~{\bf A1}--{\bf A3} be fulfilled and let \linebreak \mbox{$\varphibf_* = (\varphi_{*,1}, \ldots, \varphi_{*,p})$} be a~locally quasi-unique ground state 
 in the sense of Definition~\textup{\ref{definition-quasi-isolated-ground-state}} 
 with the corresponding Lagrange multiplier~$\Lambda_*=\diag(\lambda_{*,1},\ldots,\lambda_{*,p})$. 
 For $j =1,\ldots,p$, the eigenvalue problem \eqref{eq:evp4D2Ejj} has an~$L^2$-or\-tho\-gonal basis of eigenfunctions $v_{1,j}, v_{2,j}, \ldots$ in $T_{\varphi_{*,j}}\Sphere_{N_j}$ corresponding to the real eigenvalues $\lambda_{1,j}<\lambda_{2,j}\leq \ldots$ ordered increasingly. In addition, the smallest eigenvalue~$\lambda_{1,j}$ is simple and coincides with $\lambda_{*,j}$, i.e., $\lambda_{1,j} = \lambda_{*,j}>0$, and the corresponding eigenfunction is $v_{1,j} = \ci \varphi_{*,j}$. 
\end{proposition}
\begin{proof} 
Let $\varphibf_*\in\OB$ be a ground state. 
Note that the operator $\DDEjj(\varphibf_*)$ in \eqref{eq:second_derivative_componentwise} is symmetric and coercive on~$H_0^1(\mathcal{D}, \mathbb{C})$ and, hence, on the tangent space $T_{\varphi_{*,j}}\Sphere_{N_j} \subset H_0^1(\mathcal{D}, \mathbb{C})$. Therefore, it has countably infinite number of real eigenvalues and the corresponding eigenfunctions form an~\mbox{$L^2$-ortho}gonal basis in $T_{\varphi_{*,j}}\Sphere_{N_j}$.

For all $z\in T_{\varphi_{*,j}}\Sphere_{N_j}$, we further have 
\begin{align*}
\big\langle\DDEjj(\varphibf_*)(\ci\varphi_{*,j}), z\big\rangle
   & \!= \big\langle\calA_{\varphibf_*,j} (\ci\varphi_{*,j})\! + \!\calB_{\varphibf_*,jj}(\ci\varphi_{*,j},\varphi_{*,j}), z\big\rangle
    \!= (\lambda_{*,j}(\ci\varphi_{*,j}), z)_{L^2},
\end{align*}
which shows that $(\ci\varphi_{*,j}, \lambda_{*,j})$ is an~eigenpair to \eqref{eq:evp4D2Ejj}.
Since $\varphibf_*$ is a global minimizer of $\,\calE$ on $\OB$, the second-order necessary optimality condition implies that $\big\langle\Drm^2_{\!\varphibf\varphibf} \calL(\varphibf_*,\Lambda_*)\zbf,\zbf\big\rangle\geq 0$ for all \mbox{$\zbf\in T_{\varphibf_*}\OB$}. Then, for any $z_j\in T_{\varphi_{*,j}}\Sphere_{N_j}$, $\zbf_j = (0,\ldots, 0,z_j, 0, \ldots, 0)$ belongs to $ T_{\varphibf_*}\,\OB$ and therefore 
\[
\big\langle\DDEjj(\varphibf_*)z_j,z_j\big\rangle - (\lambda_{*,j} z_j, z_j)_{L^2}
= \big\langle\Drm^2_{\varphibf\varphibf}\calL(\varphibf_*,\Lambda_*)\zbf_j,\zbf_j\big\rangle\geq 0.
\]
This implies that $\lambda_{1,j} = \lambda_{*,j} >0$ is the smallest eigenvalue, and $v_{1,j}=\ci\varphi_{*,j}\in\calV_{\varphi_{*,j}}$ is the corresponding eigenfunction.

We prove the simplicity of the smallest eigenvalue $\lambda_{1,j} =\lambda_{*,j}$ by contradiction. Let us assume that for some \mbox{$j \in \{1,2,\ldots,p\}$}, the smallest eigenvalue $\lambda_{*,j}$ of the operator $\DDEjj(\varphibf_*) \vert_{T_{\varphi_{*,j}}\Sphere_{N_j}}$ has multiplicity greater than one, i.e., there exists an~eigenfunction $w_j \in T_{\varphi_{*,j}}\Sphere_{N_j}\setminus\mathcal{V}_{\varphi_{*,j}}$ to problem~\eqref{eq:evp4D2Ejj} such that $\big\langle\DDEjj(\varphibf_*)w_j,w_j\big\rangle - (\lambda_{*,j}w_j, w_j)_{L^2} =0$. Then for $\wbf = (0,\ldots,0,w_j,0,\ldots,0)$ belonging to $T_{\varphibf_*}\OB \setminus \calV_{\varphibf_*}$, we have 
$\big\langle\Drm^2 \calE(\varphibf_*)\wbf,\wbf\big\rangle - (\wbf \Lambda_*, \wbf)_L = 0$. This contradicts the fact that~$\varphibf_*$ is a~locally quasi-unique ground state.
\end{proof}

\begin{remark}
Assuming that $K$ is positive definite and observing that $\calB_{\varphibf_*,ij}(\ci\varphi_{*,j},\varphi_{*,i}) = 0$ for $i, j = 1, \dots, p$, a~similar result to Proposition~\ref{prop:simple_ev} holds true for the eigenpairs of the full operator $\Drm^2\calE(\varphibf_*)$ restricted to $T_{\varphibf_*}\OB$. Its smallest eigenvalue is $\lambda_{*,\text{min}} = \min_{1\leq j \leq p} \lambda_{*,j}$, and the corresponding eigenspace is spanned by $(0, \dots, 0, \ci\varphi_j, 0, \dots, 0) \in \calV_{\varphibf_*}$ for all $j$ such that \mbox{$\lambda_{*,j} = \lambda_{*,\text{min}}$}.
\end{remark}

The eigenvalue property in Proposition~\ref{prop:simple_ev} will be verified in numerical experiments. This property is also crucial to prove the coercivity of $\Drm^2_{\!\varphi_j\varphi_j} \calL(\varphibf_*,\Lambda_*)$ on the horizontal space $\Horg{\varphi_{*,j}}{g_j}$. 
\begin{proposition}[Coercivity of $\DDLjj(\varphibf_*,\Lambda_*)$] \label{prop:coercive} 
Let Assumptions~{\bf A1}--{\bf A3} be fulfilled and let~$\varphibf_*$ be a~locally quasi-unique  ground state in the sense of Definition~\textup{\ref{definition-quasi-isolated-ground-state}}  with the corresponding Lagrange multiplier~$\Lambda_*$. Then, for all $j=1,\ldots,p$, the Hessian $\DDLjj(\varphibf_*,\Lambda_*)$ is coercive on the horizontal space $\Horg{\varphi_{*,j}}{g_j}$, i.e., there exists a~constant $\alpha_j>0$ such that
\[
\big\langle\Drm^2_{\varphi_j\varphi_j} \calL(\varphibf_*,\Lambda_*)\xi,\xi\big\rangle
\geq  \alpha_j \|\xi\|_{H^1_0}^2 \qquad \mbox{for all } \quad \xi \in \Horg{\varphi_{*,j}}{g_j}. 
\]
\end{proposition}
\begin{proof} 
According to Proposition~\ref{prop:simple_ev}, we know that $\lambda_{*,j} = \lambda_{1,j}$ is the smallest (simple) eigenvalue of $ \DDEjj(\varphibf_*)$ on the tangent space $T_{\varphi_{*,j}}\Sphere_{N_j}$, and $\ci \varphi_{*,j}\in\calV_{\varphi_{*,j}}$ is the corresponding eigenfunction. Therefore, for all $\xi \in \Horg{\varphi_{*,j}}{L^2}$, we can write
\begin{equation}
\label{eq:D2E_coercivity_1}
\big\langle\DDEjj(\varphibf_*)\xi, \xi\big\rangle
\geq \lambda_{2,j} \|\xi\|_{L^2}^2
\end{equation}
with the second smallest eigenvalue $\lambda_{2,j}$ of $\DDEjj(\varphibf_*)$ on $T_{\varphi_{*,j}}\Sphere_{N_j}$. 
As $\calV_{\varphi_{*,j}}$ and $\Horg{\varphi_{*,j}}{g_j}$ are two closed subspaces with trivial intersection, there exists a constant $0 \leq C_{g_j} < 1$ such that $C_{g_j} =\sup_{\xi \in \Horg{\varphi_{*,j}}{g_j}\setminus \{0\}} \frac{(\ci\varphi_{*,j}, \xi)_{L^2}}{N_j^{1/2}\|\xi\|_{L^2}}$. Using that $\ci\varphi_{*,j}$ is an eigenfunction of $\DDEjj(\varphibf_*)$ to $\lambda_{1,j}$ together with \eqref{eq:D2E_coercivity_1}, we obtain for all $\xi \in \Horg{\varphi_{*,j}}{g_j}$ that 
\begin{align*}
\big\langle\DDEjj(\varphibf_*)\xi, \xi\big\rangle
& \geq \lambda_{1,j}(\ci\varphi_{*,j}, \xi)_{L^2}^2/N_j + \lambda_{2,j} \big(\|\xi\|_{L^2}^2 - (\ci\varphi_{*,j}, \xi)^2_{L^2}/N_j\big) \\
& \geq \big(\lambda_{2,j} - C_{g_j}^2(\lambda_{2,j} - \lambda_{1,j})\big)\|\xi\|_{L^2}^2,
\end{align*}
and therefore $\langle \Drm^2_{\!\varphi_j\varphi_j}\calL(\varphibf_*,\Lambda_*)\xi, \xi \rangle\geq (1 - C_{g_j}^2)(\lambda_{2,j}-\lambda_{1,j})\|\xi\|_{L^2}^2$.
Now along the same lines as in the single-component case \cite[Lemma 2.3]{HenY25}, we can establish the required coercivity. 
\end{proof}

We conclude this section with the following proposition, which establishes the inf–sup stability of the whole Hessian $\Drm^2_{\!\varphibf\varphibf}\calL(\varphibf_*,\Lambda_*)$ on~$\Horg{\varphibf_*}{g}$. This property will play a~central role in the subsequent local convergence analysis.

\begin{proposition}[Inf-sup stability of $\Drm^2_{\!\varphibf \varphibf}\calL(\varphibf_*,\Lambda_*)$]\label{prop:infsup}
Let Assumptions~{\bf A1}--{\bf A3} be fulfilled and let $K$ be positive definite. Further, let $\varphibf_*  \in \OB$ be a locally quasi-unique ground state  with the corresponding Lagrange multiplier $\Lambda_* = \diag(\lambda_{*,1}, \dots, \lambda_{*,p})$.
Then the operator $\Drm^2_{\varphibf \varphibf}\calL(\varphibf_*,\Lambda_*)$ is inf-sup stable on $\Horg{\varphibf_*}{g}$, meaning that there exists a constant $\beta >0$ such that
\begin{equation}\label{eq:infsup}
\myinf_{\vphantom{d}\zetabf \in \Horg{\varphibf_*}{g}\setminus\{\zerobf\}}\; \sup_{\xibf \in \Horg{\varphibf_*}{g}\setminus\{\zerobf\}}\frac{\langle \Drm^2_{\!\varphibf \varphibf}\calL(\varphibf_*,\Lambda_*)\xibf, \zetabf \rangle}{\|\xibf\|_{\Hsp} \| \zetabf\|_{\Hsp}} \geq \beta.
\end{equation}
\end{proposition}
\begin{proof}
    As $\Drm^2_{\!\varphibf \varphibf}\calL(\varphibf_*,\Lambda_*)$ is bounded on $\Hsp$ and positive definite on $\Horg{\varphibf_*}{g}$, we conclude that the operator $\Drm^2_{\varphibf \varphibf}\calL(\varphibf_*,\Lambda_*) : \Horg{\varphibf_*}{g} \rightarrow(\Horg{\varphibf_*}{g})^{\star}$ is bounded and injective. Moreover,  for all $\xibf\in \Horg{\varphibf_*}{g}$, we have
    \begin{equation}
\big\langle \Drm^2_{\varphibf \varphibf}\calL(\varphibf_*,\Lambda_*)\xibf , \xibf \big\rangle  
= \big\langle \calA_{\varphibf_*} \xibf + \calB_{\varphibf_*} (\xibf,\varphibf_*), \xibf \big\rangle 
- (\xibf \Lambda_* , \xibf)_\Lsp \\
 \nonumber \geq \alpha \|\xibf\|^2_{\Hsp} - \|\Lambda_{*}\|_2 \|\xibf\|^2_{\Lsp}
\end{equation}
 with a constant $\alpha>0$, i.e., the operator $\Drm^2_{\varphibf \varphibf}\calL(\varphibf_*,\Lambda_*)$ satisfies a G{\aa}rding inequality on $\Horg{\varphibf_*}{g}$. Therefore, the Fredholm alternative guarantees that $\Drm^2_{\varphibf \varphibf}\calL(\varphibf_*,\Lambda_*)$ is invertible and has a bounded inverse \cite{Spe15}. Then the inf-sup condition~\eqref{eq:infsup} follows from \cite[Th.~2.1.44]{SautS11}. 
\end{proof}


\section{Riemannian gradient-type methods and their convergence analysis} \label{sec:RGD_loc}
In the context of optimization on Riemannian manifolds, the choice of a~Riemannian metric plays a~central role. Using the Hamiltonian~$\calA_\varphibf$ to define an~energy-adaptive metric depending on the current iteration naturally incorporates a~problem-specific structure into Riemannian optimization algorithms and acts as an~effective form of preconditioning. This idea has previously been explored with great success in ground state computations \cite{AHPS25,AltPS22,HenP20,HenY25}. To further accelerate convergence, it is appealing to incorporate second-order information into the metric for a~stronger preconditioning effect \cite{FenT25,GaoPY25,MisS16}. In this section, we present the Riemannian gradient method for a~broad class of metrics on both the generalized oblique manifold $\OB$ as well as the quotient manifold~$\Seg$ and analyze its local convergence properties. We then introduce two particular metrics, resulting in the energy-adaptive and Lagrangian-based Riemannian gradient descent methods, which correspond to the two ideas outlined above.

%
%
\subsection{Riemannian gradient descent method}\label{ssec:RGD}
We start by equipping the generalized ob\-lique manifold $\OB$ with a Riemannian metric $g_\varphibf$ which satisfies  the following assumptions:
\vspace{0.01em}
\begin{itemize}[itemsep=0.2em]
\item[\bf{B1:}] $g_{\varphibf}$ is defined on $\Hsp$, bounded and coercive with respect to the $\Hsp$-norm;
\item[\bf{B2:}] $g_\varphibf$ is a~phase invariant product metric as defined in \eqref{eq:prod_metric} and \eqref{eq:rotinvar};
\item[\bf{B3:}] $g_\varphibf$ is horizontally compatible with the $\Lsp$-metric at every ground state $\varphibf_*$ in the sense of Definition~\ref{def:hor_compat}. 
\end{itemize}
\vspace{0.4em}
The last assumption may appear rather restrictive. However, in Sections~\ref{ssec:eaRGD} and~\ref{ssec:LagrRGD} below, we will demonstrate that it is satisfied for two relevant and practically important metrics.

Similarly to the non-rotating case \cite{AHPS25}, we compute the \emph{Riemannian gradient} of the energy functional $\,\calE$ at $\varphibf\!\in\!\OB$ with respect to~$g_\varphibf$~as
\[
   \grad_g{\calE}(\varphibf) 
    = \calP_\varphibf^g(\calG_\varphibf^{-1}\calA_\varphibf\varphibf)
        = \calG_{\varphibf}^{-1}\big(\calA_\varphibf\varphibf - \varphibf\,\Sigma_g(\varphibf)\big),
\]
where $\calP_\varphibf^g$ is as in \eqref{eq:projection}, and $\Sigma_g(\varphibf) = \out{\varphibf}{\calG_{\varphibf}^{-1}\calA_\varphibf \varphibf}\out{\varphibf}{\calG_{\varphibf}^{-1}\varphibf}^{-1} \in \DpR$. It will play a~central role in the design and analysis of the optimization schemes.

Let $\calH_\varphibf^g$ denote the horizontal space at $\varphibf$ with respect to the metric $g_\varphibf$. We can use horizontal lifts in $\calH_\varphibf^g$ to define an~induced metric on the quotient manifold~$\,\Seg$. 
Indeed, Proposition~\ref{prop:phase_lift} ensures that the following induced Riemannian metric on $\,\Seg$ is well defined:
\begin{equation*}
    \widetilde g_{[\varphibf]}(\widetilde\ybf, \widetilde\zbf) := g_{\varphibf}(\widetilde\ybf_\varphibf^{\hrm,g},\widetilde\zbf_\varphibf^{\hrm,g}) \qquad \text{for all }\, \widetilde\ybf,\widetilde\zbf\in T_{[\varphibf]}\Seg,
\end{equation*}
where $\widetilde\ybf_\varphibf^{\hrm,g},\widetilde\zbf_\varphibf^{\hrm,g} \in \Horg{\varphibf}{g}$ are the horizontal lifts
of $\widetilde\ybf$ and $\widetilde\zbf$, respectively.
The horizontal lift of the Riemannian gradient $\grad_{\widetilde g}\widetilde{\calE}([\varphibf]) \in T_{[\varphibf]}\Seg$ with respect to this metric is then given by 
\[
    \big( \grad_{\widetilde g}\widetilde{\calE}([\varphibf]) \big)_{\varphibf}^{\hrm,g} 
        = \grad_g\calE(\varphibf).
\]
We can utilize this representation to perform optimization on the quotient manifold $\,\Seg$ via its counterparts on the generalized oblique manifold $\OB$. Any optimization algorithm that converges to a~minimizer on $\Seg$ naturally yields a~ground state \mbox{$\varphibf_*\! \in\! \OB$} by selecting a representative of the resulting equivalence class. In particular, convergence to a~strict local minimum ensures that the solution $\varphibf_*$ is locally quasi-unique in the sense of Definition~\ref{definition-quasi-isolated-ground-state}. In practice, all computations are carried out on (a~discretized version of) the generalized oblique manifold~$\OB$.

We now introduce a~\emph{retraction}, which allows movement in a~tangent direction while remaining on the manifold. Given $\varphibf\in\OB$ and $\zbf\in T_\varphibf\OB$, the simplest choice for a retraction on $\OB$ is the normalization of the components of $\varphibf + \zbf$, i.e., 
\begin{equation}\label{eq:retraction}
\calR_\varphibf(\zbf) = \calN(\varphibf + \zbf) 
\end{equation}
with the normalization operator $\calN(\vbf) = \vbf\out{\vbf}{\vbf}^{-1/2}N^{1/2}$ for all $\vbf \in \Hsp\setminus\{0\}$. This directly induces a~retraction on the quotient space $\,\Seg$ via the horizontal lift, which reads
\begin{equation}\label{eq:retraction_quotient}
    \widetilde{\calR}_{[\varphibf]}(\widetilde\zbf)=\pi\big(\calR_{\varphibf}(\widetilde\zbf_\varphibf^{\hrm,g})\big),  \qquad \widetilde\zbf \in T_{[\varphibf]}\,\Seg. 
\end{equation}
Note that this retraction is well defined due to Proposition~\ref{prop:phase_lift}.

The \emph{Riemannian gradient descent} (RGD) \emph{method} on the quotient manifold $\,\Seg$ is given by
\begin{equation}\label{eq:gRGD_quotient}
    [\varphibf_{k+1}] = \widetilde\calR_{[\varphibf_k]}\big( -\tau_k\, \grad_{\widetilde g}\widetilde \calE([\varphibf_k]) \big)
\end{equation}
with the step size $\tau_k > 0$. 
Using \eqref{eq:retraction_quotient}, this iteration can also be written as
\begin{equation*}
    \pi(\varphibf_{k+1}) = \pi\Big(\calR_{\varphibf_k}\Big( -\tau_k\, \big(\grad_{\widetilde g}\widetilde \calE([\varphibf_k])\big)_{\varphibf_k}^{\hrm,g} \Big)\Big) = \pi\Big(\calR_{\varphibf_k}\big( -\tau_k\, \grad_g\calE(\varphibf_k) \big)\Big).
\end{equation*}
As a~consequence, we can equivalently run the RGD iteration on the generalized oblique manifold $\OB$ given by 
\begin{equation}\label{eq:gRGD}
    \varphibf_{k+1} 
    = \calR_{\varphibf_k}\big( -\tau_k\, \grad_g\calE(\varphibf_k) \big) 
    = \calN\big(\varphibf_k -\tau_k\, \calG_{\varphibf_k}^{-1}\big(\calA_{\varphibf_k}\varphibf_k - \varphibf_k\Sigma_g(\varphibf_k)\big)\big).
\end{equation}

 Next, we extend the quantitative local convergence results from \cite{FenT25,HenP20,HenY25} for the single-component case and from \cite{AHPS25} for non-rotating multicomponent models, developed for an energy-adaptive metric, to rotating multicomponent BECs and to a~general metric $g_\varphibf$

%
%
\subsection{Local convergence analysis} \label{ssec:LocConv}
We now aim to characterize the behavior of the RGD scheme \eqref{eq:gRGD} in a~neighborhood of a~ground state and to determine its convergence rate. This section may be viewed as an~extension of \cite{FenT25} to the multicomponent setting and to the case where no additional compactness assumptions are imposed. The main idea behind our analysis is to use the convergence theory for fixed point iterations, in particular Ostrowski’s theorem \cite{Shi81}. To apply this theory, we exploit the correspondence between the phase-induced local non-uniqueness in the minimization problem~\eqref{eq:min} on $\OB$ and the local uniqueness of solutions to problem \eqref{eq:min_quotient} formulated on the quotient manifold $\Seg$. This is achieved by adopting a~fix-the-phase strategy through an~auxiliary iteration.

First, we observe that the RGD iteration \eqref{eq:gRGD} with a~constant step size $\tau_k=\tau$ can be written as the fixed point iteration $\varphibf_{k+1} = \varPsi_\tau(\varphibf_k)$, where
\begin{equation*}
\varPsi_\tau(\varphibf)
 = \calN\big(\widehat{\varPsi}_\tau(\varphibf)\big), 
\qquad
\widehat{\varPsi}_\tau(\varphibf) 
 = \varphibf - \tau\, \calG_{\varphibf}^{-1} \big(\calA_{\varphibf}\,\varphibf - \varphibf \,\Sigma_g(\varphibf) \big).
\end{equation*}
A ground state $\varphibf_*\in\OB$ is a fixed point of both $\widehat{\varPsi}_\tau$ and $\varPsi_\tau$, as we have \mbox{$\widehat{\varPsi}(\varphibf_*)\! = \varphibf_*\! = \varPsi(\varphibf_*)$}. Moreover, from Assumption~\textbf{B2}, it follows that $\varPsi_\tau(\varphibf\Theta)=\varPsi_\tau(\varphibf)\Theta$ for all $\Theta\in\DpS$. This implies, in particular, that $\varphibf_*\Theta$ is also a fixed point of $\varPsi_\tau$.

Similarly to \cite{AHPS25}, we calculate the directional derivatives of $\varPsi_\tau$ and $\widehat{\varPsi}_\tau$ at $\varphibf_*$ along $\vbf\in\Hsp$~as
\begin{align}
\Drm \varPsi_\tau(\varphibf_*)\vbf
& = \calP_{\varphibf_*}^\Lsp\big(\Drm \widehat{\varPsi}_\tau(\varphibf_*)\vbf\big), \label{eq:DPsi} \\
\Drm \widehat{\varPsi}_\tau(\varphibf_*)\vbf
& = \vbf 
    - \tau\, \calP_{\varphibf_*}^g\big(\calG_{\varphibf_*}^{-1}\Drm^2_{\!\varphibf\varphibf} \calL(\varphibf_*,\Lambda_*)\vbf\big), \label{eq:der_hatPsi}       
\end{align}
where $\Lambda_* = \Sigma_g(\varphibf_*)$ is the Lagrange multiplier corresponding to $\varphibf_*$.
For $\vbf=\ci\varphibf_*$, we obtain $\Drm \varPsi_\tau(\varphibf_*)(\ci\varphibf_*)=\ci\varphibf_*$, implying that $\ci\varphibf_*$ is an~eigenfunction of~$\Drm \varPsi_\tau(\varphibf_*)$ corresponding to the eigenvalue $1$. This is a~direct consequence of \eqref{eq:D2_lambda} and shows that the spectral radius of the $\R$-Fréchet derivative $\Drm\varPsi_\tau(\varphibf_*)$ cannot be less than $1$. Hence, unlike the non-rotational case \cite{AHPS25}, we cannot directly apply Ostrowski’s theorem to estimate the local convergence rate.
However, noticing that \mbox{$\ci\varphibf_* \in \calV_{\varphibf_*}$} and recalling that the vertical space $\calV_{\varphibf_*}$ contains all directions that stay within equivalence classes, we turn back to the RGD scheme \eqref{eq:gRGD_quotient} on the quotient manifold $\,\Seg$ and introduce the following auxiliary iteration.
\begin{definition}
\label{def:auxi-scheme-definition} 
Let $\varphibf_{*}\! \in\! \OB$ be a locally quasi-unique ground state. For all \mbox{$\vbf\! \in\! \Hsp$}~such that $\out{\vbf}{\varphibf_*}_\C$ is invertible, set $\Theta_{\varphibf_*}\!(\vbf) = \overline{\out{\vbf}{ \varphibf_*}}_\C\big|\out{\vbf}{\varphibf_*}_\C\big|^{-1} \in\DpS$ and \mbox{$\Gamma_{\varphibf_*}\!(\vbf)\! = \vbf\,\Theta_{\varphibf_*}\!(\vbf)$}. Then, for a given initial value $\psibf_0 \in \OB$, the auxiliary iteration is defined by
\begin{equation}\label{eq:auxi-scheme}
\psibf_{k+1} = \Upsilon_\tau(\psibf_k)
\end{equation}
with the auxiliary fixed point mapping $\Upsilon_\tau(\vbf) = \Gamma_{\varphibf_*}\!(\varPsi_\tau(\vbf))$.
\end{definition}
\begin{remark}\label{remark:auxi-scheme}
Note that $\Gamma_{\varphibf_*}$ serves as a fixed phase representation of equivalence classes in $\,\Seg$. To be precise, for $\psibf, \varphibf \in \OB$, the relation $\Gamma_{\varphibf_*}\!(\psibf) = \Gamma_{\varphibf_*}\!(\varphibf)$ is equivalent to $[\psibf] = [\varphibf]$, and it holds that $\Gamma_{\varphibf_*}\!(\psibf) \in [\psibf]$ as well as $\Gamma_{\varphibf_*}\!(\Gamma_{\varphibf_*}\!(\psibf)) = \Gamma_{\varphibf_*}\!(\psibf)$. Furthermore, iteration \eqref{eq:auxi-scheme} is well-defined in a~neighborhood of a~ground state $\varphibf_*$ due to $\Theta_{\varphibf_*}(\varphibf_*)=I_p,$ and among all phase shifts within $[\varphibf_*]$, the operator $\Upsilon_{\tau}$ admits the unique fixed point $\varphibf_{*}$.
\end{remark}

If the sequences $\{\varphibf_{k} \}_{k=0}^\infty$ and $\{\psibf_{k} \}_{k=0}^\infty$ generated by \eqref{eq:gRGD} and \eqref{eq:auxi-scheme}, respectively, are initialized with the same initial data, i.e., $\varphibf_{0} = \psibf_{0}$, and if $\out{\varphibf_k}{\varphibf_*}_\C$ is invertible for every $k \geq 1$, then due to~\textbf{B2}, we have     
    $\psibf_{k} = \Gamma_{\varphibf_*}\!(\varphibf_{k})$.
By Remark~\ref{remark:auxi-scheme}, this relation implies that $[\psibf_k] = [\varphibf_k]$ for all $k \geq 0$. This means that the auxiliary iteration \eqref{eq:auxi-scheme} is also a representation of the RGD iteration \eqref{eq:gRGD_quotient} on the quotient manifold $\,\Seg$, and we can directly transfer local convergence results for \eqref{eq:auxi-scheme} to \eqref{eq:gRGD}.

To analyze the local convergence of the auxiliary iteration \eqref{eq:auxi-scheme}, we need to compute the spectral radius of the $\R$-Fr\'echet derivative $\Drm\!\Upsilon_\tau(\varphibf_*)$, which determines the convergence rate. Using~\eqref{eq:DPsi} and~\eqref{eq:der_hatPsi}, we calculate for all $\vbf\in\Hsp$
\begin{align}
\Drm \!\Upsilon_\tau(\varphibf_*)\vbf & = \calP_{\ci\varphibf_*}^\Lsp(\calP_{\varphibf_*}^\Lsp(\Drm \widehat{\varPsi}_\tau(\varphibf_*)\vbf)) = \varPi_{\varphibf_*}^{\Lsp}\vbf
    - \tau\,  \varPi_{\varphibf_*}^g\calG_{\varphibf_*}^{-1}\Drm^2_{\!\varphibf\varphibf} \calL(\varphibf_*,\Lambda_*)\vbf, \label{eq:der_Uspilon}
\end{align}
where $\varPi_{\varphibf_*}^\Lsp$ and $\varPi_{\varphibf_*}^g$ are the $\Lsp$- and $g_{\varphibf_*}$-orthogonal projections on the horizontal space \mbox{$\Horg{\varphibf_*}{g} \!=\! \Horg{\varphibf_*}{\Lsp}$}. 
According to \cite[Prop.~VII.6.7]{Con10}, the spectral radius $\varrho_{\tau,*}:=\varrho(\Drm\!\Upsilon_\tau(\varphibf_*))$ is determined by the approximate point spectrum of $\Drm \!\Upsilon_\tau(\varphibf_*)$ consisting of all $\mu_i$ for which there exists a~sequence $\{\vbf_{i,k}\}_{k \in \N}\subset \Hsp$ with $\|\vbf_{i,k}\|_{\Hsp} = 1$ such that  
\begin{equation}\label{eq:app_spectrum}
    \lim_{k\to\infty} \|
    \Drm \!\Upsilon_\tau(\varphibf_*)\vbf_{i,k} - \mu_i\,\vbf_{i,k}
    \|_\Hsp = 0.
\end{equation}
As $\Pi_{\varphibf_*}^\Lsp\Drm \!\Upsilon_\tau(\varphibf_*) = \Drm \!\Upsilon_\tau(\varphibf_*)$, it is sufficient to consider $\vbf_{i,k} \in \Horg{\varphibf_*}{g} = \Horg{\varphibf_*}{\Lsp}$, and by Assumption~{\bf B1}, we can replace the $\Hsp$-norm in \eqref{eq:app_spectrum} by the one induced by $g_{\varphibf_*}$. Further, since $\Drm \!\Upsilon_\tau(\varphibf_*)$ is $g_{\varphibf_*}$-symmetric, the approximate eigenvalues are real.
Overall, for any approximate eigenvalue $\mu_i\in\R$, there exists a~$g_{\varphibf_*}$-normalized sequence $(\wbf_{i,k})_{k \in \N} \subset \Horg{\varphibf_*}{g}$ such that for any $g_{\varphibf_*}$-bounded sequence $\{\xibf_k\}_{k\in\N} \subset \Horg{\varphibf_*}{g}$, we obtain that
\begin{align}
    0 =& \lim_{k \to \infty}
    | \langle \calG_{\varphibf_*}\Drm \!\Upsilon_\tau(\varphibf_*)\wbf_{i,k}, \xibf_k\rangle
    - \mu_i \langle \calG_{\varphibf_*}\wbf_{i,k}, \xibf_k \rangle| \notag \\
    =& \lim_{k \to \infty}
    | \langle \calG_{\varphibf_*} \wbf_{i,k}, \xibf_k \rangle 
    - \tau\langle \Drm^2_{\!\varphibf\varphibf} \calL(\varphibf_*,\Lambda_*)\wbf_{i,k}, \xibf_k \rangle
     - \mu_i \langle \calG_{\varphibf_*}\wbf_{i,k}, \xibf_k \rangle| \notag \\
     = & \lim_{k \to \infty} |\langle\Drm^2_{\!\varphibf\varphibf} \calL(\varphibf_*,\Lambda_*)\wbf_{i,k},\xibf_k\rangle 
     - \eta_i \langle \calG_{\varphibf_*}\wbf_{i,k}, \xibf_k \rangle| \label{eq:auxAppEVP},
\end{align}
where $\eta_i = (1 - \mu_i)/\tau$. Choosing $\xibf_k = \wbf_{i,k}$ yields the relation
\[
        \lim_{k\to\infty} \frac{\langle\Drm^2_{\!\varphibf\varphibf} \calL(\varphibf_*,\Lambda_*)\wbf_{i,k},\wbf_{i,k}\rangle}{\langle \calG_{\varphibf_*} \wbf_{i,k}, \wbf_{i,k} \rangle} = \eta_i.
\]
Propositions~\ref{prop:D2Eproperties} and~\ref{prop:infsup} together with Assumption~{\bf B1} imply that there are constants $c_{\varphibf_*}, C_{\varphibf_*} > 0$ independent of $\wbf_{i,k}$ such that $c_{\varphibf_*} < \eta_i < C_{\varphibf_*}$.

We are now ready to prove that the RGD method \eqref{eq:gRGD} converges locally linear in $H$ to a~ground state $\varphibf_*$ and to specify its convergence rate.
\begin{theorem}[Local convergence rate for RGD]\label{th:LocConv}
Let Assumptions~\textup{\textbf{A1}}--\textup{\textbf{A3}} and \textup{\textbf{B1}}--\textup{\textbf{B3}} be fulfilled, and let $K$ be positive definite. Further, let \mbox{$\varphibf_*\in\OB$} be a locally quasi-unique ground state, and let $\eta_{\inf} = \inf_i \eta_i > c_{\varphibf_*}$ and $\eta_{\sup} = \sup_i \eta_i < C_{\varphibf_*}$, where $\eta_i$ are solutions to~\eqref{eq:auxAppEVP}.
Then the spectral radius $\varrho_{\tau,*} = \varrho(\Drm\!\Upsilon_\tau(\varphibf_*))$ satisfies 
    \begin{equation}\label{eq:specrad}
    \varrho_{\tau,*} = \sup_i |1 - \tau\eta_i|  = \max\{1 - \tau \eta_{\inf},\, \tau\eta_{\sup} - 1\}.
    \end{equation}
In particular, $\varrho_{\tau,*} < 1$ for all $\tau \in (0, \tfrac{2} {\eta_{\sup}})$. In that case, for every $\epsilon>0$, there exists a~neighborhood~$\,\mathcal{U}_\epsilon$ of $\varphibf_*$ in $\OB$ and a positive constant $C_\epsilon$ such that for all initial functions $\varphibf_0\in\mathcal{U}_\epsilon$, the~sequence $\{\varphibf_k\}_{k=0}^\infty$ generated by the RGD method \eqref{eq:gRGD} with the constant step size $\tau_k=\tau$ fulfills
\[
    \inf_{\Theta\in \D(p,\SSS)}\|\varphibf_k \Theta-\varphibf_*\|_{\Hsp}
    \leq C_\epsilon\, |\varrho_{\tau,*}+\epsilon|^k\, \|\varphibf_0-\varphibf_*\|_{\Hsp}\qquad \text{for all } k\geq 1,
\]
meaning that the RGD iteration \eqref{eq:gRGD} converges locally linear with rate~$\varrho_{\tau,*}$.   
\end{theorem}
\begin{proof}
 By the discussion above, if $\mu_i$ is in the approximate point spectrum of $\Drm\!\Upsilon_\tau(\varphibf_*)$, then $\eta_i = (1 - \mu_i)/\tau$ solves \eqref{eq:auxAppEVP} for some $g_{\varphibf_*}$-normalized sequence $\{\wbf_{i,k}\}_{k\in\N}$.
 Since $\eta_i > c_{\varphibf_*} > 0$, we get $\rho_{\tau,*} < 1$ for all $\tau\in(0,\frac{2}{\eta_{\sup}})$.  Therefore, the local convergence estimate for the auxiliary iteration \eqref{eq:auxi-scheme} follows from Ostrowski's theorem \cite{Shi81} and can be stated as 
\begin{align*}
    \|\psibf_k-\varphibf_*\|_{\Hsp}
    \leq C_\epsilon\, |\varrho_{\tau,*}+\epsilon|^k\, \|\psibf_0-\varphibf_*\|_{\Hsp}\qquad \text{for all } k\geq 1,
    \end{align*}
Together with the relations $\psibf_0=\varphibf_0$ and  $\psibf_{k} = \Gamma_{\varphibf_*}(\varphibf_k)=\varphibf_{k}\, \Theta_{\varphibf_*}\!(\varphibf_{k})$, this estimate yields
    \begin{align*}
        \inf_{\Theta \in \D(p,\SSS)}\| \varphibf_{k} \Theta - \varphibf_*\|_{\Hsp} &\leq \| \varphibf_{k}\,\Theta_{\varphibf_*}\!(\varphibf_{k}) - \varphibf_* \|_{\Hsp} = \|\psibf_k-\varphibf_*\|_{\Hsp} \\
        & \leq  C_\epsilon\, |\varrho_{\tau,*}+\epsilon|^k\, \|\varphibf_0-\varphibf_*\|_{\Hsp}.
    \end{align*}
This completes the proof.    
\end{proof}

Under the additional assumption that $\calG_{\varphibf_*}^{-1}\calA_{\varphibf_*}$ is a~compact perturbation of the identity, we can replace the approximate eigenvalue problem \eqref{eq:auxAppEVP} by a~standard one. This condition, which resembles \cite[Assum.~{\bf(A6)}(iv)]{FenT25}, guarantees that $\calG_{\varphibf_*}$ captures enough information from~$\calA_{\varphibf_*}$.

\begin{proposition}[Point spectrum]\label{prop:point_spectrum}
Let Assumptions~\textup{\textbf{A1}}--\textup{\textbf{A3}} and \textup{\textbf{B1}}--\textup{\textbf{B3}} be fulfilled, $K$ be positive definite, and let $\varphibf_*\in\OB$ be a~locally quasi-unique ground state. If $\calG_{\varphibf_*}^{-1}\calA_{\varphibf_*} - \calI$ is compact, then the spectral radius of $\Drm\!\Upsilon_\tau(\varphibf_*)$ is determined only by its eigenvalues.
\end{proposition}
\begin{proof}
By \eqref{eq:der_Uspilon}, we have  
\[
  \Drm\!\Upsilon_\tau(\varphibf_*) = \varPi_{\varphibf_*}^{\Lsp}\vbf 
  - \tau\, \varPi_{\varphibf_*}^g\vbf
  - \tau\,  \varPi_{\varphibf_*}^g\big((\calG_{\varphibf_*}^{-1}\Drm^2_{\!\varphibf\varphibf} \calL(\varphibf_*,\Lambda_*) - \calI)\vbf\big).
\]
Recall that the spectrum of a compact operator consists only of its eigenvalues and zero and that a projection of a compact operator remains compact. As all approximate eigenfunctions of $\Drm\!\Upsilon_\tau(\varphibf_*)$ lie in the horizontal space $\Horg{\varphibf_*}{\Lsp}=\Horg{\varphibf_*}{g}$, projections of the identity only shift the spectrum. Therefore, in order to show that the spectral radius of $\Drm\!\Upsilon_\tau(\varphibf_*)$ is only affected by its eigenvalues, it suffices to show that the operator $\calG_{\varphibf_*}^{-1}\Drm_{\varphibf\varphibf}^2\calL(\varphibf_*,\Lambda_*) - \calI$ is compact. Since $\Drm_{\varphibf\varphibf}^2\calL(\varphibf_*,\Lambda_*)\vbf = \calA_{\varphibf_*}\vbf + \calB_{\varphibf_*}(\vbf,\varphibf_*) - \vbf\Lambda_*$ for all $\vbf \in \Hsp$ and $\calG_{\varphibf_*}^{-1}\calA_{\varphibf_*} - \calI$ is compact by assumption, it remains to verify that $\vbf \mapsto \calG_{\varphibf_*}^{-1}(\calB_{\varphibf_*}(\vbf,\varphibf_*) - \vbf\Lambda_*)$ is compact. We can estimate
\begin{align*}
    \|\calG_{\varphibf_*}^{-1}(\calB_{\varphibf_*}(\vbf,\varphibf_*) - \vbf\Lambda_*)\|_\Hsp & 
    \leq C_g \|(\calB_{\varphibf_*}(\vbf,\varphibf_*) - \vbf\Lambda_*)\|_{\Hsp^{\star}} \\
    & \leq C_g \sup_{\|w\|_\Hsp = 1}\langle\calB_{\varphibf_*}(\vbf,\varphibf_*), \wbf \rangle 
    + C_gC_2\|\Lambda_*\|_2\|\vbf\|_L 
    \leq C \|\vbf\|_{[L^4(\calD,\C)]^p},
\end{align*}
with constants $C_g,C_2$ and $C$ independent of $\vbf$.
Compactness now follows as in \cite[Lem.~A.1]{HenY25}.
\end{proof}

In particular, we can replace \eqref{eq:auxAppEVP} in Theorem \ref{th:LocConv} by the following eigenvalue problem: find an~eigenfunction $\vbf_i \in \Horg{\varphibf_*}{g}$ and an~eigenvalue $\eta_i \in \R$ such that
\[
    \big\langle \Drm_{\varphibf\varphibf}^2\calL(\varphibf_*,\Lambda_*)\vbf_i,\xibf\big\rangle 
    = \eta_i\langle\calG_{\varphibf_*}\vbf_i,\xibf\rangle
\qquad \text{for all \;}\xibf \in \Horg{\varphibf_*}{g}.
\]

Having established and quantified local convergence for a~general RGD method, in the next two subsections, we consider two specific metrics and apply this convergence result to the associated RGD schemes.

%
%
\subsection{Energy-adaptive Riemannian gradient descent method}\label{ssec:eaRGD}

A very natural choice of metric is the energy-adaptive metric $g_{\varphibf,a}(\ybf,\zbf) = a_\varphibf(\ybf,\zbf)=\langle\calA_\varphibf\, \ybf,\zbf\rangle$ for all $\ybf, \zbf \in T_\varphibf\OB$. It obviously satisfies Assumptions~\textbf{B1} and~\textbf{B2}. Moreover, Assumption~\textbf{B3} is fulfilled due Proposition~\ref{prop:hor_compat_eigvalue}. This choice of metric results in the energy-adaptive Riemannian gradient 
\begin{equation}\label{eq:eagradE}
\grad_a\calE(\varphibf)=\varphibf - \calA_{\varphibf}^{-1}\varphibf\out{\varphibf}{\calA_{\varphibf}^{-1}\varphibf}^{-1}N.
\end{equation}
Then the \emph{energy-adaptive Riemannian gradient descent} (eaRGD) \emph{iteration} is given by 
\begin{equation}\label{eq:eaRGD}
    \varphibf_{k+1} = \calN\big((1-\tau_k)\varphibf_k + \tau_k\,\calA_{\varphibf_k}^{-1}\varphibf_k\out{\varphibf_k}{\calA_{\varphibf_k}^{-1}\varphibf_k}^{-1}N\big).
\end{equation}
This method can be interpreted as a~damped version of the inverse subspace iteration applied to the operator $\calA_\varphibf$, where the adaptive damping is governed by the step size $\tau_k$. In the special case $\tau_k=1$, it reduces to the standard inverse subspace iteration, commonly known as the $\calA$-method, see \cite{Hen23,HenP20}.

Applying Theorem~\ref{th:LocConv} and Proposition~\ref{prop:point_spectrum} (where $\calG_{\varphibf_*}^{-1} = \calA_{\varphibf_*}^{-1}$, so the additional compactness assumption is trivially fulfilled), we find that the convergence of \eqref{eq:eaRGD} to a locally quasi-unique ground state $\varphibf_*$ is governed by the following eigenvalue problem: find an eigenfunction $\vbf_{a,i} \in \Horg{\varphibf_*}{a}$ and an eigenvalue $\eta_{a,i} \in \R$ such that
\begin{equation*}
 \big\langle \Drm_{\varphibf\varphibf}^2\calL(\varphibf_*,\Lambda_*)\vbf_{a,i}, \xibf\rangle = \eta_{a,i}\langle\calA_{\varphibf_*}\vbf_{a,i},\xibf\rangle
\qquad \text{for all }\xibf \in \Horg{\varphibf_*}{a}.
\end{equation*}
Here, $\Horg{\varphibf_*}{a}$ denotes the horizontal space at $\varphibf_*$ with respect to $g_{\varphibf_*,a}$. Using $\vbf_{a,i}$ as a test function, we estimate
\begin{align*}
\eta_{a,i} & 
    <  1+\frac{\langle\calB_{\varphibf_*}(\vbf_{a,i},\varphibf_*),\vbf_{a,i}\rangle}{\langle\calA_{\varphibf_*}\vbf_{a,i},\vbf_{a,i}\rangle} 
    = 1 + \frac{2\int_\mathcal{D}\big(\re(\varphibf_*\circ\overline{\vbf}_{a,i})\big) K \big( \re(\varphibf_*\circ\overline{\vbf}_{a,i})\big)^T\dx}{\|\vbf_{a,i}\|_{\sR}^2+\int_\mathcal{D}(\varphibf_*\circ\overline{\varphibf}_*)K(\vbf_{a,i}\circ\overline{\vbf}_{a,i})^T\dx} < 3.
\end{align*}
The last inequality follows from 
\begin{align*}
\int_\mathcal{D}\big(\re(\varphibf_*\circ\overline{\vbf}_{a,i})\big) K \big( \re(\varphibf_*\circ\overline{\vbf}_{a,i})\big)^T\dx
\le \int_{\mathcal{D}} (\varphibf_*\circ\overline{\varphibf}_*)K(\vbf_{a,i}\circ\overline{\vbf}_{a,i})^T\dx,
\end{align*} 
which can be obtained from Assumption~\textbf{A3} and Young's inequality. Therefore, we can guarantee the convergence rate $\varrho_{\tau,a,*} = \sup_{i\in\mathbb{N}}|1 - \tau\eta_{a,i}| < 1$ for any $\tau\in(0, \frac{2}{3})$, implying local convergence. Moreover, if all eigenvalues $\eta_{a,i} < 1$, then the constant step size $\tau = 1$, which corresponds to the basic inverse iteration, is admissible.

In addition to local convergence, we can even establish the global convergence of the densities $|\varphibf_k|^2$ in the eaRGD method \eqref{eq:eaRGD}, where the squared absolute value is taken component-wise. 

\begin{theorem}[Global convergence for eaRGD]\label{thm:global_convergence}
Let Assumptions~{\bf A1}--{\bf A3} be fulfilled. Let a~sequence $\{\varphibf_k\}_{k=0}^\infty\subset \OB$ be generated by the eaRGD method \eqref{eq:eaRGD}. Then there exist constants $C_K,C_0>0$ such that for any step size \mbox{$0 < \tau_{\min} \leq \tau_k \leq \tau_{\max} \leq (1 + \frac{9}{2}C_KC_0^2)^{-1} < 1$} and any $k \geq 0$, the following relations hold:
\begin{itemize}
\item[{\rm(i)}] There exists a limit energy $\calE_\infty:=\lim\limits_{k \rightarrow \infty} \calE(\varphibf_k)$.

\item[{\rm(ii)}] There is a subsequence $\{ \varphibf_{k_l} \}_{l=0}^{\infty}$ of the sequence $\{ \varphibf_k \}_{k =0}^{\infty}$ and $\varphibf_* \in \OB$ such that $\lim\limits_{l\rightarrow \infty} \| \varphibf_{k_l} - \varphibf_*\|_{H}=0$. Furthermore, $\varphibf_*$ is a~constrained critical point of the energy functional $\,\calE$ with the Lagrange multiplier $\Lambda_*=\coout{ \calA_{\varphibf_*} \varphibf_* }{\varphibf_*}N^{-1}$.

\item[{\rm(iii)}] If an~accumulation point $\varphibf_*$ of $\{\varphibf_k\}_{k=0}^\infty$ is a locally quasi-unique ground state to \eqref{eq:min}, then the whole sequence of densities $\{|\varphibf_k|^2\}_{k=0}^\infty$ converges to the corresponding ground state density $|\varphibf_*|^2$ with
\begin{align*}
\lim_{k\rightarrow \infty} \| \, |\varphibf_k|^2 - |\varphibf_*|^2 \,\|_{\Lsp} = 0 
\qquad
\mbox{and}
\qquad
\lim_{k\rightarrow \infty} \| \, |\varphibf_k| - |\varphibf_*| \,\|_{\Hsp} = 0.
\end{align*}
\end{itemize}
\end{theorem}

\begin{proof}
    The proof can be found in Appendix~\ref{app:proof_global}.
\end{proof}

This theorem provides useful insights into the behavior of the eaRGD method \eqref{eq:eaRGD}. Under appropriate step size selection, it ensures the convergence of a~subsequence of the iterates $\varphibf_k$ to a~constrained critical point starting from any initial guess. This makes the eaRGD scheme particularly suitable as a~reliable initialization strategy for other (local) optimization methods. Note that the convergence for the full sequence $\varphibf_k$ cannot be guaranteed due to multiplicative phase shifts. However, considering the associated densities $|\varphibf_k|^2$, the phase ambiguity is removed, and the entire sequence $|\varphibf_k|^2$ is proven to converge, provided that the constrained critical point is a~locally quasi-unique ground state. 

%
%
\subsection{Lagrangian-based gradient descent method}\label{ssec:LagrRGD}

In~\cite{AHPS25}, the Lagrangian-based RGD method has been introduced for non-rotating multicomponent BECs following ideas from \cite{GaoPY25,MisS16}, which relies on the idea of block diagonal preconditioning and exploits second-order information on the energy functional and constraints. This method usually provides a~better convergence rate and computational performance than eaRGD once it is initialized sufficiently close to a~ground state.

Consider a~\emph{regularized Lagrangian} 
\[
\calL_\omega(\varphibf,\Lambda)=\calE(\varphibf)-\frac{\omega}{2}\trace\big(\Lambda(\out{\varphibf}{\varphibf}-N)\big)
\]
with a~regulari\-zation parameter $\omega\in(0,1)$. We define a~\emph{Lagrangian-based metric} 
\[
g_{\varphibf,\omega}(\ybf,\zbf) 
= \langle\calG_{\varphibf,\omega}\ybf,\zbf\rangle
= \sum_{j=1}^p \langle\calG_{\varphibf,\omega,j}y_j,z_j\rangle,
\qquad \ybf,\zbf\in T_\varphibf\OB,
\]
where the component operators $\calG_{\varphibf,\omega,j}:H_0^1(\DC)\to H^{-1}(\DC)$ are given by 
\begin{align*}
\langle\calG_{\varphibf,\omega,j}v_j,w_j\rangle
& = \big\langle\Drm^2_{\varphi_j\varphi_j}\calL_{\omega}(\varphibf,\Lambda)v_j,w_j\big\rangle \\
& = \big\langle\calA_{\varphibf,j}v_j+\calB_{\varphibf,jj}(v_j,\varphi_j)-\omega \lambda_j v_j,w_j\big\rangle, \qquad v_j,w_j\in H_0^1(\calD,\C),
\end{align*}
with $\lambda_j=\langle\calA_{\varphibf,j}\varphi_j,\varphi_j\rangle/N_j$ for $j= 1,\ldots,p$. Introducing an~operator $\calB_{\varphibf}^{\drm}:\Hsp\times\Hsp\to\Hsp^\star$ defined as 
\[
\langle\calB_{\varphibf}^{\drm}(\vbf,\ubf), \wbf\rangle 
= 2 \int_{\calD} \re(\varphibf\circ\overline\vbf)\Diag(K)\re(\ubf\circ\overline\wbf)\dx
= \sum_{j=1}^p \langle\calB_{\varphibf,jj}(u_j,v_j), w_j\rangle, 
 \quad \ubf,\vbf,\wbf\in\Hsp,
\]
where $\Diag(K)$ denotes a diagonal matrix with the same diagonal elements as $K$, we can write $\langle\calG_{\varphibf,\omega}\vbf,\wbf\rangle=\langle\calA_{\varphibf}\vbf+\calB_{\varphibf}^{\drm}(\vbf,\varphibf)-\omega\,\vbf\Lambda,\wbf\rangle$. It follows from Propositions~\ref{prop:aphi} and~\ref{prop:D2Eproperties} that the operator $\calG_{\varphibf,\omega}$ is symmetric, bounded, and phase invariant. Moreover, for sufficiently small~$\omega$, 
it is coercive and therefore invertible. Therefore, the metric $g_{\varphibf,\omega}$ fulfills Assumptions~{\bf B1} and~{\bf B2}. Assumption~{\bf B3} immediately follows from Propositions~\ref{prop:hor_compat_eigvalue} and~\ref{prop:simple_ev}.

The corresponding \textit{Lagrangian-based Riemannian gradient descent} (LagrRGD) is defined by 
\begin{equation}\label{eq:lagrRGD}
    \varphibf_{k+1} = \calN\big(\varphibf_k - \tau_k\,\calG_{\varphibf_k,\omega}^{-1}(\calA_{\varphibf_k}\varphibf_k - \varphibf_k\out{\varphibf_k}{\calG_{\varphibf_k,\omega}^{-1}\calA_{\varphibf_k}\varphibf_k}\out{\varphibf_k}{\calG_{\varphibf_k,\omega}^{-1}\varphibf_k}^{-1})\big).
\end{equation}
To assess its local convergence to a~locally quasi-unique ground state $\varphibf_*$, we apply Theorem~\ref{th:LocConv} and Proposition~\ref{prop:point_spectrum}. The additional compactness assumption follows directly from the relation \mbox{$\calG_{\varphibf_*,\omega}^{-1}\calA_{\varphibf_*} -\calI = \calG_{\varphibf_*,\omega}^{-1}(\calA_{\varphibf_*}-\calG_{\varphibf_*,\omega})$}, where the operator on the right-hand side can be shown to be compact similarly to the proof of Proposition~\ref{prop:point_spectrum}. Therefore, we can consider the following eigenvalue problem: find an~eigenfunction $\vbf_{\omega,i} \in \Horg{\varphibf_*}{\omega}$ and an~eigenvalue $\eta_{\omega,i} \in \R$ such that
\begin{equation}\label{eq:weighted_EVP_lagrange_new}
 \big\langle \Drm^2_{\!\varphibf\varphibf}\calL(\varphibf_*,\Lambda_*)\vbf_{\omega,i},\xibf\big\rangle 
 = \eta_{\omega,i}\langle\calG_{\varphibf_*,\omega}\vbf_{\omega,i},\xibf\rangle
\qquad \text{for all }\xibf \in \Horg{\varphibf_*}{\omega}, 
\end{equation}
where $\Horg{\varphibf_*}{\omega}$ denotes the horizontal space at $\varphibf_*$ with respect to the metric $g_{\varphibf_*,\omega}$.  For the LagrRGD method~\eqref{eq:lagrRGD}, the convergence rate 
\begin{equation}\label{eq:LagrRGDrho}
    \varrho_{\tau,\omega,*} = \sup_{i\in\mathbb{N}}\big|1-\tau\eta_{\omega,i}\big|
\end{equation}
obeys the inequality $\varrho_{\tau,\omega,*} < 1$ only for step sizes $0< \tau < \frac{2}{\eta_{\omega,\sup}}$, where $\eta_{\omega,\sup}$ denotes the largest eigenvalue of \eqref{eq:weighted_EVP_lagrange_new}. 

\begin{remark}[Single-component case] \label{rem:conv_rate}
In the single-component case $p=1$, for a locally quasi-unique ground state $\varphi_*$ and the corresponding Lagrange multiplier $\lambda_*$, we obtain from~\eqref{eq:weighted_EVP_lagrange_new} that
\begin{align}
1-\eta_{\omega,i}
    & = \frac{(1-\omega)\lambda_*\|v_{\omega,i}\|_{L^2}^2}{\langle\Drm^2\calE(\varphi_*)v_{\omega,i}-\omega\lambda_*v_{\omega,i},v_{\omega,i}\rangle}>0 \label{eq:prop_eta}
\end{align}
provided $\omega\in(0,1)$. This implies that $\eta_{\omega,i}<1$ and, in particular, that $\tau=1$ is an~admissible step size. Then, taking into account that $v_{\omega,i}\in\calH_{\varphi_*}^{\omega}$ and that $\calH_{\varphi_*}^{\omega}$ coincides with the \mbox{$L^2$-orth}ogonal complement of the vertical space $\calV_{\varphi_*}$ in $T_{\varphi_*}\Sphere_{N_1}$, we derive from \eqref{eq:prop_eta} by using Proposition~\ref{prop:simple_ev} that 
\[
\varrho_{\tau=1,\omega,*} 
    = \sup_{i\in\mathbb{N}}\big|1-\eta_{\omega,i}\big|=\frac{1-\omega}{\lambda_2/\lambda_1-\omega} 
    < \frac{\lambda_1}{\lambda_2}<1,
\]
where $\lambda_1=\lambda_*$ is the smallest and $\lambda_2$ is the second smallest eigenvalue of $\Drm^2\calE(\varphi_*)\big|_{T_{\varphi_*}\Sphere_{N_1}}$. 
Considering the convergence rate $\varrho_{\tau=1,\omega,*}$ as a function of the regularization parameter $\omega$, we observe that this function decreases as $\omega$ increases. This suggests choosing $\omega$ close to $1$ to ensure fast convergence of the LagrRGD method~\eqref{eq:lagrRGD}. In this case, however, the resulting operator $\calG_{\varphi_k,\omega}$ may become ill-conditioned or even lose coercivity. Similar behavior is expected for $p > 1$. This computational issue will be further examined in numerical experiments.
\end{remark}
%

\section{Numerical experiments}\label{sec:numerics}
In this section, we examine three different BEC models on a~2D spatial domain with two and three components to showcase the theoretical results and further investigate the proposed optimization algorithms. These models were discretized using a finite element scheme with bi-quadratic elements, though other approaches, such as spectral elements or finite differences, could also be employed. The implementation was done in the Julia programming language using the package {\tt Ferrite.jl} for the finite element discretization. The source code for these experiments is available in the GitHub repository 
\begin{center}
{\tt\url{https://www.github.com/MaHermann/Riemannian-coupledGPE-rotation}}.    
\end{center}

All three models are considered on the spatial domain $\calD = [-10, 10]^2$, which has been uniformly partitioned with mesh width \mbox{$h = \tfrac{20}{64} = 0.3125$}, resulting in $64 = 2^6$ elements per direction and $n=16641$ degrees of freedom.
The iterations of the optimization schemes are terminated once the (discrete) residual
\[
r_k = \Big(\sum_{j=1}^p \overline{R(\phi_{k,j})}^T M R(\phi_{k,j})\Big)^{1/2}
\]
falls below $10^{-14}$ or after a maximal number of iterations. Here, 
\mbox{$R(\phi_{k,j})= (A_{\Phi_k,j}- \lambda_{k,j} M)\phi_{k,j}$} are the component residuals, $\phi_{k,j}$ is the $j$-th column of the discrete wave function $\Phi_k\in\C^{n\times p}$, $M$ is the $L^2$-mass matrix, $A_{\Phi_k,j}$ is the stiffness matrix for the $j$-th component, and $\lambda_{k,j}=\overline{\phi}_{k,j}^TA_{\Phi_k,j}\phi_{k,j}/N_j$.
To allow a fair comparison of both schemes, eaRGD and LagrRGD, they were initialized with the same initial guess computed by the eaRGD method until the residual satisfies $r_k<10^{-2}$ starting with a~normalized non-zero constant function having zero values on the boundary. 

For solving linear systems, we use the preconditioned conjugate gradient (CG) method with the preconditioner determined from the incomplete LU decomposition of $A_{0,j}$. In inner iterations, the stopping criterion for the relative residuals is based on an~adaptive tolerance $r_k\cdot\text{tol}_{\text{CG}}$ with $\text{tol}_{\text{CG}}=10^{-8}$ for the two-component models and $\text{tol}_{\text{CG}}=10^{-1}$ for the three-component model.

In Riemannian optimization, choosing an~appropriate step size $\tau_k$ is critical for ensuring convergence and efficiency of the algorithm. In our experiments, we tested different step size strategies. 

\begin{itemize}[leftmargin=15pt]
\item \textit{Fixed step size}: In Section~\ref{ssec:LocConv}, we have discussed the local convergence of RGD with a~constant choice $\tau_k = \tau$. We take a~particularly simple step size $\tau = 1$ as a~baseline, which turns out to be admissible for all test models, ensuring convergence of eaRGD and LagrRGD without the need for further tuning. 

\item \textit{Exact line search}:
Alternately, we can determine the locally optimal step size 
\[
\tau_k^\text{LS} = \argmin_{\tau \in [\tau_\text{min}, \tau_\text{max}]} \calE\big(\calR_{\varphibf_k}(- \tau \grad_g\calE(\varphibf_k))\big)
\]
with appropriately chosen lower and upper bounds $0<\tau_\text{min} <\tau_\text{max}$. In our experiments, we take $\tau_{\min} = 0.1$ and $\tau_{\max} = 10$. 
Similarly to \cite[App.~B]{HenY25}, we can compute the optimal value of~$\tau_k^\text{LS}$ at step $k$ comparatively cheaply by explicitly expanding the weighted mass matrices.

\item \textit{Adaptive step size}: A simple but effective step size strategy has been introduced in \cite{MalM20} and extended to RGD in~\cite{AnsM25}. We find that even a simplified version of the Euclidean strategy accelerates convergence significantly and comes at almost no computational cost. The corresponding step size given by
\[
    \tau_k^\text{AD} = \frac{\| \varphibf_k - \varphibf_{k-1} \|_{\Lsp}}{\|\grad_g\calE(\varphibf_k) - \grad_g\calE(\varphibf_{k-1}) \|_{\Lsp}}
\]
can be computed at low cost, since it circumvents the repeated evaluations of the energy functional that are typically needed in line search algorithms.
\end{itemize}

%
%
\subsection{Two-component BEC models}
In the first set of numerical experiments, we study two models with two 
interacting components: a weakly interacting model (Model~1) and a strongly interacting model (Model~2) with interaction matrices 
\[
K_{\rm w} = \begin{bmatrix}120 & 20 \\ 20 & 100 \end{bmatrix} \qquad\text{ and } \qquad
K_{\rm s} = \begin{bmatrix}120 & 60 \\ 60 & 100 \end{bmatrix},
\]
respectively. Other physical parameters are the same for both models. These are the number of particles $N_1=2.0$ and $N_2=1.0$, the angular frequencies $\Omega_1 = -1.0$ and $\Omega_2 = -1.2$, and the harmonic potentials 
\[
    V_1(x, y) = \tfrac{3}{4}\big((0.8x)^2 + (1.2y)^2 \big), \qquad  
    V_2(x, y) = \tfrac{1}{2}\big((1.2x)^2 + (0.9y)^2 \big).
\]

{\em Model~1}: The component densities of the local minimizer $\varphibf_*$ computed by the eaRGD method are presented in Figure~\ref{fig:groundstates_ex1}. The corresponding energy is $\calE(\varphibf_*) = 8.4864708$, and the Lagrange multipliers are given by $\lambda_{*,1} = 9.6417158$ and $\lambda_{*,2} = 4.5589610$. Table~\ref{tab:Esecphi_spectrum_exp1} with selected eigenvalues (ordered increasingly) of the component operators $\calA_{\varphibf_*,j}$ and the projected component Hessians $\calF_{\varphibf_*,j}=\Drm^2_{\varphi_j\varphi_j}\calE(\varphibf_*)\big|_{T_{\varphi_{*,j}}\Sphere_{N_j}}$ demonstrates that each Lagrange multiplier coincides with the smallest eigenvalue of the corresponding projected Hessians but not of~$\calA_{\varphibf_*,j}$, although it is present in the spectrum of~$\calA_{\varphibf_*,j}$. This confirms the results stated in Proposition~\ref{prop:simple_ev}. We also verified that all the eigenvalues of $\Drm_{\!\varphibf\varphibf}^2\calL(\varphibf_*,\Lambda_*)$ restricted to $\calH_{\varphibf_*}^{\Lsp}$ are positive, which shows that the computed constrained critical point is indeed locally quasi-unique.

\begingroup
\setlength{\belowcaptionskip}{0pt}
\begin{figure}[t]
	\centering
\includegraphics{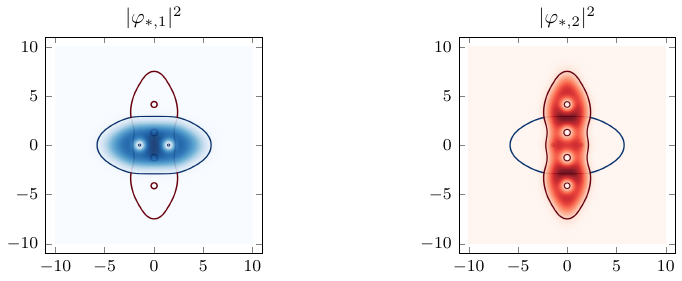}
\caption{Model 1: component densities of the computed local minimizer.}
	\label{fig:groundstates_ex1}
\end{figure}   
\endgroup

\begingroup
\setlength{\belowcaptionskip}{0pt}
\begin{table}[t]
    \centering
    \begin{tabular}{|c|c|c|c|c|c|}
        \hline
        $i$ & 1 & 7 & \cellcolor{lgray}8 & 9 & \cellcolor{lgray}10 \\ \hline
        $\lambda_{i}(\mathcal{A}_{\varphibf_*,1})$ & 7.9246460 & 9.6103155 & \cellcolor{lgray} 9.6417158 & 9.7017540 & 9.8916715 \\ \hline
        $\lambda_{i}(\mathcal{A}_{\varphibf_*,2})$ & 4.0456313 & 4.52922557 & 4.5374526 & 4.5520104 & \cellcolor{lgray} 4.5589610  \\ \hline
    \multicolumn{6}{c}{} \\[-2mm]
         \hline
        $i$ & \cellcolor{lgray}1 & 2 & 3 & 4 & 5 \\ \hline
        $\lambda_i(\calF_{\varphibf_*,1})$ & \cellcolor{lgray}9.6417158 &  9.6719120 & 9.6762225 & 9.7126806 & 9.8638330 \\ \hline
		$\lambda_i(\calF_{\varphibf_*,2})$ & \cellcolor{lgray}4.5589610 & 4.5648379 & 4.5821349 & 4.5857457 & 4.5960499 \\ \hline
    \end{tabular}
    \caption{Model~1: (top) selected eigenvalues of the component operators $\mathcal{A}_{\varphibf_*,j}$ for $j=1,2$;  (bottom) five smallest eigenvalues of $\calF_{\varphibf_*,j}=\Drm^2_{\varphi_j\varphi_j}\calE(\varphibf_*)\big|_{T_{\varphi_{*,j}}\Sphere_{N_j}}$ for $j=1,2$.}
    \label{tab:Esecphi_spectrum_exp1}
\end{table}
\endgroup

\begingroup
\setlength{\abovecaptionskip}{0pt}
\setlength{\belowcaptionskip}{0pt}
\begin{figure}[h!t]
\centering
\includegraphics{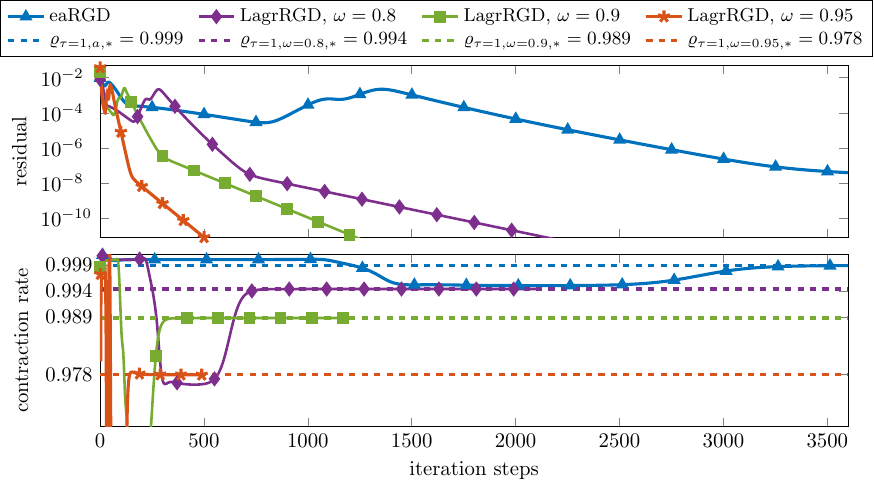}
\caption{Model 1: (top) convergence histories of the residuals; (bottom) contraction and convergence rates for different methods with $\tau=1$.}
\label{fig:mod1_res_rate}
\end{figure}
\endgroup

\begingroup
\setlength{\belowcaptionskip}{0pt}
\begin{figure}[h!t]
\centering
\includegraphics{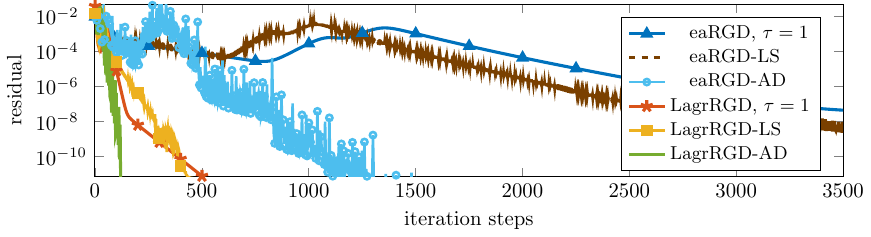}
\caption{Model 1: convergence histories of the residuals for eaRGD and \mbox{LagrRGD} with \mbox{$\omega=0.95$} for different step size strategies.}
\label{fig:mod1_stepsize}
\end{figure}
\endgroup

The convergence histories of the residuals for eaRGD and LagrRGD for different values of $\omega$ with a~fixed step size $\tau=1$ are shown in Figure~\ref{fig:mod1_res_rate}~(top). One can see that LagrRGD converges faster than eaRGD. Moreover, as $\omega$ increases, the number of the LagrRGD iterations required to achieve the prescribed tolerance decreases. 
This behavior is expected, since the convergence rate $\varrho_{\tau=1,\omega,*}$
decreases with increasing $\omega$. In Figure~\ref{fig:mod1_res_rate}~(bottom), we compare the eaRGD and LagrRGD convergence rates $\varrho_{\tau=1,a,*}$ and $\varrho_{\tau=1,\omega,*}$ given in \eqref{eq:specrad} and \eqref{eq:LagrRGDrho}, respectively, with the contraction rate defined as $\|\varphibf_{k+1}-\varphibf_*\|_{\Hsp}/\|\varphibf_{k}-\varphibf_*\|_{\Hsp}$.
We observe that the contraction rate asymptotically approaches the corresponding predicted convergence rate. This numerically validates the local convergence results in Sections~\ref{ssec:eaRGD} and~\ref{ssec:LagrRGD}.  

In Figure~\ref{fig:mod1_stepsize}, we present the convergence histories of the residuals for earRGD and LagrRGD with a fixed step size $\tau=1$, a step size determined by an exact line search, and an adaptive step size strategy. While employing an exact line search reduces the total number of iterations compared to the fixed step size, the adaptive approach yields the best overall performance, achieving faster convergence across all examples.

{\em Model~2}: 
Comparing the component densities of the local minimizer for Model~2 presented in Figure~\ref{fig:groundstates_ex2} with those for Model~1 in Figure~\ref{fig:groundstates_ex1}, we observe that for both models, the densities have smooth profiles with vortex cores visible as dips. The bulk region of each condensate component is primarily governed by its respective trapping potential, and their spatial overlap is determined by the inter-component interaction regime. The vortex number and the resulting vortex lattice patterns depend on the masses, angular velocities, and interaction strengths. As the inter-component interaction strength $\kappa_{12}$ increases in Model~2, the system undergoes a transition towards phase separation, resulting in the formation of distinct domains in the second component. 

\begingroup
\setlength{\abovecaptionskip}{4pt}
\setlength{\belowcaptionskip}{0pt}
\begin{figure}[ht]
	\centering
    \includegraphics{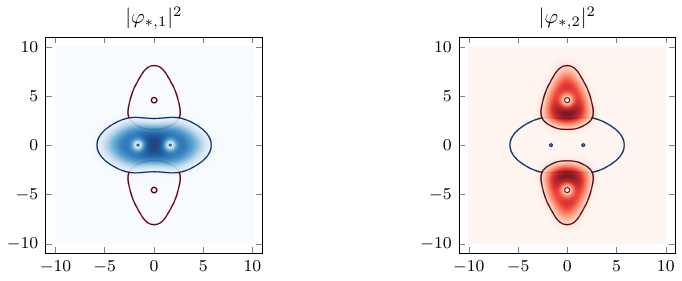}
	\caption{Model 2: component densities of the computed local minimizer.}
	\label{fig:groundstates_ex2}
\end{figure}
\endgroup

\begingroup
\setlength{\abovecaptionskip}{0pt}
\setlength{\belowcaptionskip}{0pt}
\begin{figure}[h!t]
\centering
\includegraphics{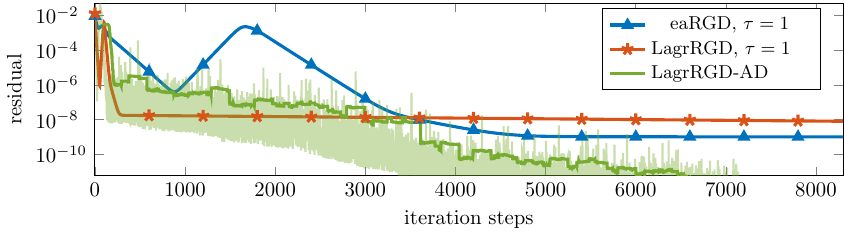}
\caption{Model 2: 
convergence histories of the residuals for eaRGD and \mbox{LagrRGD} with \mbox{$\omega=0.95$} for different step size strategies. For LagrRGD-AD, we also include a moving mean with a~window of 100 steps in each direction.
}
\label{fig:stepsize_ex2}
\end{figure}
\endgroup

The energy value is $\calE(\varphibf_*) = 8.6821968$ and the Lagrange multipliers are $\lambda_{*,1} = 9.6544660$ and \mbox{$\lambda_{*,2}=5.0782488$}. They coincide with the smallest eigenvalues of the projected Hessians~$\calF_{\varphibf_*,j}$ and are present in the spectrum of the component operators $\calA_{\varphibf_*,j}$, as shown in Table~\ref{tab:Esecphi_spectrum_exp2}. It should be noted that the relative spectral gap of $\calF_{\varphibf_*,2}$, defined as $\Delta_2 = (\lambda_2(\calF_{\varphibf_*,2}) - \lambda_1(\calF_{\varphibf_*,2}))/\lambda_1(\calF_{\varphibf_*,2})$, is very small.
As a~result, the residuals stagnate at a~level of about $10^{-8}$ for both eaRGD and LagrRGD when the fixed step size of $\tau=1$ is used; see Figure~\ref{fig:stepsize_ex2}. However, combining LagrRGD with the adaptive step size strategy effectively overcame this stagnation, yielding smaller residual values and improved convergence. The same behavior was also observed for eaRGD.

\begingroup
\setlength{\belowcaptionskip}{0pt}
\begin{table}[h!t]
    \centering
    \begin{tabular}{|c|c|c|c|c|c|}
        \hline
        $i$ & 1 & \cellcolor{lgray} 7 & 8 & 9 & \cellcolor{lgray} 10 \\ \hline
        $\lambda_{i}(\mathcal{A}_{\varphibf_*,1})$ & 7.8667998 & \cellcolor{lgray} 9.6544660 & 9.7021386 & 9.8445635 &  10.0504973 \\ \hline
        $\lambda_{i}(\mathcal{A}_{\varphibf_*,2})$ & 4.5250052 & 5.0511228 & 5.0517242 & 5.0782014 & \cellcolor{lgray}  \hphantom{1}5.0782488 \\ \hline
    \multicolumn{6}{c}{} \\[-2mm]
            \hline
        $i$ & \cellcolor{lgray} 1 & 2 & 3 & 4 & 5 \\ \hline
        $\lambda_i(\calF_{\varphibf_*,1})$ & \cellcolor{lgray} 9.6544660 & 9.6779311 & 9.6840046 & 9.7261413 & 9.8075549 \\ \hline
		  $\lambda_i(\calF_{\varphibf_*,2})$ & \cellcolor{lgray} 5.0782488 & 5.0782810 & 5.1160543 & 5.1162083 & 5.1547135 \\\hline
    \end{tabular}
    \caption{Model~2: (top) selected eigenvalues of the component operators $\mathcal{A}_{\varphibf_*,j}$ for $j=1,2$; (bottom) five smallest eigenvalues of $\calF_{\varphibf_*,j}=\Drm^2_{\varphi_j\varphi_j}\calE(\varphibf_*)\big|_{T_{\varphi_{*,j}}\Sphere_{N_j}}$ for $j=1,2$.}
    \label{tab:Esecphi_spectrum_exp2}
\end{table}
\endgroup
%
%
\subsection{Three-component model}

Next, we consider the three-component model (Model 3) with the following physical parameters: 
\[
\begin{array}{lll}
N_1=2.0, & \quad\Omega_1=-1.0, & \quad V_1(x, y) = \tfrac{1}{2}\big((0.9x)^2 + (1.1y)^2 \big), \\[2mm]
N_2=1.0, & \quad\Omega_2=-1.1, & \quad V_2(x, y) = \tfrac{1}{2}\big((1.1x)^2 + (0.9y)^2 \big), \\[2mm]
N_3=3.0, & \quad\Omega_3=-1.2, & \quad V_3(x, y) = \tfrac{1}{2}\big(x^2 + y^2 \big) + \sin{x}^2 + \tfrac{1}{2}\sin{y}^2,
\end{array} \hspace{-5mm}\quad K = \begin{bmatrix}100 & 40 & 50 \\ 40 & 125 & 60 \\ 50 & 60 & 150 \end{bmatrix}.
\]
The component densities of the local minimizer $\varphibf_*$ computed by the eaRGD method are presented in Figure~\ref{fig:groundstates_ex3}. We observe that the first and second components partially overlap, while the third component exhibits a~nearly shell-like configuration, almost entirely surrounding the other two components. 
As for two-component models, the smallest eigenvalues of $\calF_{\varphibf_*,j}$ coincide with the Lagrange multipliers corresponding to $\varphibf_*$. Figure~\ref{fig:mod3_convergence} presents the convergence histories for the residuals for different optimization schemes with the fixed step size of $\tau=1$. Compared to the two-component models, the number of iterations required for convergence increases significantly, possibly due to the more complex interplay of components, which results in a~more challenging optimization problem. By employing an~adaptive step size strategy, the number of iterations can be further reduced, thereby improving the performance of the optimization schemes. For example, for LagrRGD with $\omega=0.95$, a residual of $10^{-11}$ is achieved in approximately $800$~iterations.

\begin{figure}[t]
	\centering
    \includegraphics{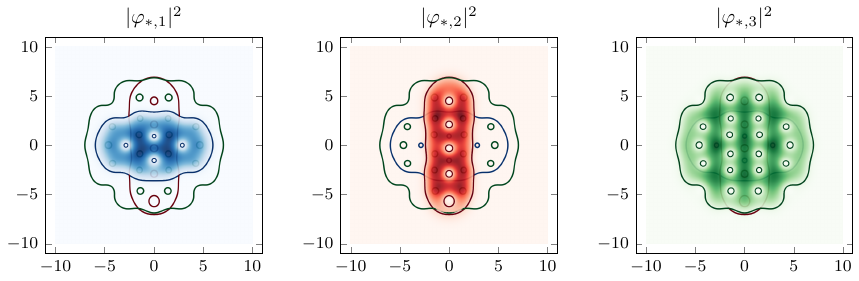}
	\caption{Model~3: component densities of the computed ground state.}
	\label{fig:groundstates_ex3}
\end{figure}

\begingroup
\setlength{\abovecaptionskip}{0pt}
\setlength{\belowcaptionskip}{0pt}
\begin{figure}[h!t]
\centering
\includegraphics{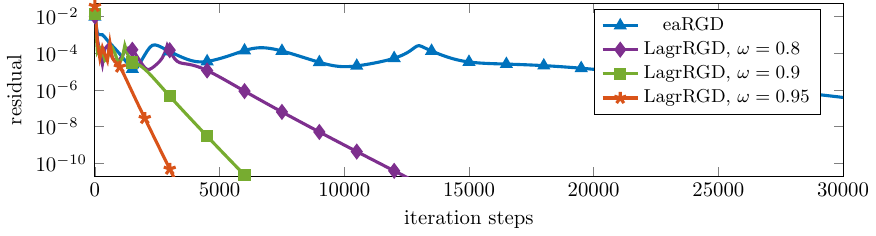}
\caption{Model 3: convergence histories of the residuals for different methods with $\tau=1$.}
\label{fig:mod3_convergence}
\end{figure}
\endgroup

%
%
\subsection{\texorpdfstring{Choosing $\omega$ for the Lagrangian-based Riemannian gradient descent}{Choosing omega for the Lagrangian-based Riemannian gradient descent}}

The re\-gu\-larization parameter $\omega$ in the LagrRGD method affects the overall runtime in two opposing ways. On the one hand, larger values of $\omega$ lead to a lower outer iteration count, as shown theoretically for a~single-component system in Remark~\ref{rem:conv_rate} and confirmed for multicomponent systems by the numerical experiments above. On the other hand, $\omega$ also impacts the condition number of the stiffness matrix of the operator $\calG_{\varphibf_k,\omega}$, which in turn affects the computational cost per iteration. To better understand this trade-off and to identify guidelines for choosing $\omega$ effectively in practice, we examine the convergence behavior and conditioning of linear systems for different values of $\omega$.

For an Hermitian positive definite matrix $A$ and the $L^2$-mass matrix $M$, the condition number $\kappa_M(A)$ of the matrix pair $(A,M)$ is defined as the quotient of its largest and smallest eigenvalues. Then for a~ground state $\Phi_*$ and the stiffness matrix of the component Hessian $\Drm^2_{\varphi_j\varphi_j}\calE(\varphibf_*)$ given by  $H_{\Phi_*,j}=A_{\Phi_*,j} + B_{\Phi_*,jj}$, the diagonal blocks of the matrix version of the Lagrangian-based weighting operator $\calG_{\varphibf_*,\omega}$ can be written as $G_{\Phi_*,\omega,j} = H_{\Phi_*,j} - \omega\lambda_{*,j} M$ with the Lagrange multiplier \mbox{$\lambda_{*,j}=\overline{\phi}_{*,j}^TA_{\Phi_*,j}\phi_{*,j}^{}/(\overline{\phi}_{*,j}^TM\phi_{*,j}^{})$} for $j=1,\ldots,p$. Then the eigenvalues of the matrix pairs $(G_{\Phi_*,\omega,j},M)$ and $(H_{\Phi_*,j},M)$, which are assumed to be ordered increasingly, are related by 
$\lambda_{i,j}(G_{\Phi_*,\omega,j},M) = \lambda_{i,j}(H_{\Phi_*,j},M) - \omega\lambda_{*,j}$.
We also know that $\lambda_{*,j}$ is an eigenvalue of $(H_{\Phi_*,j},M)$, and hence $\lambda_{*,j}\geq\lambda_{1,j}(H_{\Phi_*,j},M)$. Therefore, we have \vspace*{-2mm}
\begin{align}
\kappa_M(G_{\Phi_*,\omega,j}) 
    & = \frac{\lambda_{n,j}(G_{\Phi_*,\omega,j},M)}{\lambda_{1,j}(G_{\Phi_*,\omega,j},M)} 
    = \frac{\lambda_{n,j}(H_{\Phi_*,j},M) - \omega\,\lambda_{*,j}}{\lambda_{1,j}(H_{\Phi_*,j},M) - \omega\,\lambda_{*,j}} \notag\\
   &  \geq  \frac{\lambda_{n,j}(H_{\Phi_*,j},M) - \omega\,\lambda_{1,j}(H_{\Phi_*,j},M)}{\lambda_{1,j}(H_{\Phi_*,j},M) - \omega\,\lambda_{1,j}(H_{\Phi_*,j},M)} 
    = \frac{1}{1 - \omega}\big(\kappa_M(H_{\Phi_*,j}) - \omega\big). \vspace*{-2mm}
    \label{eq:kappaMG}
\end{align} 
This shows that for $\kappa_M(H_{\Phi_*,j}) \gg  \omega$, the condition number of $G_{\Phi_*,\omega,j}$ grows like $1/(1 - \omega)$ as~$\omega$ approaches~$1$. As the involved quantities depend continuously on $\Phi$, this behavior remains valid in a neighborhood of the ground state $\Phi_*$, which is also supported numerically in Figure~\ref{fig:condition}.

\begingroup
\setlength{\belowcaptionskip}{0pt}
\begin{figure}[t]
\centering
\includegraphics{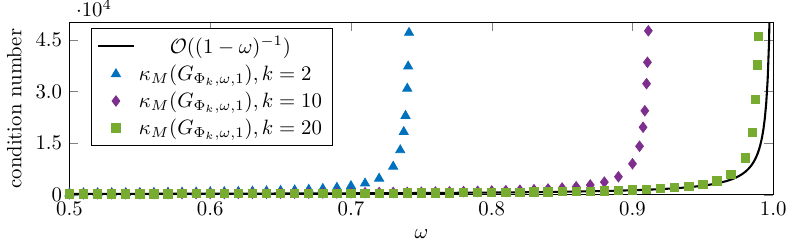}    
\caption{Condition numbers of $G_{\Phi_k,\omega,1}$ of Model 2 for different values of $\omega$. Each $\Phi_k$ has been obtained by $k$ steps of eaRGD from a constant initial value. Since the same method is used to initialize LagrRGD, this illustrates why for increasing values of $\omega$, more  iteration steps are needed in the initialization phase. Note that the condition numbers are shown only for positive definite~$G_{\Phi_k,\omega,1}$.}
\label{fig:condition}
\end{figure}
\endgroup

In each step of LagrRGD, we need to solve two linear systems involving $G_{\Phi_k,\omega,j}$ for \mbox{$j=1,\ldots,p$}, cf.~\eqref{eq:lagrRGD}. When employing preconditioned CG, a~higher condition number results in a larger number of inner CG iterations per outer iteration of LagrRGD. To simplify the analysis, we assume that this number of inner iterations scales as $\calO(1/\sqrt{1-\omega})$, which is true for CG without preconditioning. Note that in practice, as described above, a preconditioner is typically used, which significantly reduces the iteration count and potentially affects the precise dependence on $\omega$. In the single-component case, we can therefore estimate the total cost of the LagrRGD algorithm~as 
\[
    \calO(\text{\#outer iterations}\cdot\text{\#inner iterations}) 
    = \calO\bigg(\Big(\log\Big(1+\frac{\Delta}{1 - \omega}\Big) \sqrt{1 - \omega}\Big)^{-1}\bigg),
\]
where we used Theorem~\ref{th:LocConv} (ignoring the dependence on $\epsilon$), Remark~\ref{rem:conv_rate} and estimate~\eqref{eq:kappaMG}. Here, $\Delta = (\lambda_2 - \lambda_1) / \lambda_1$ with the eigenvalues $\lambda_1$ and~$\lambda_2$ as in Remark~\ref{rem:conv_rate}.

Assuming that a similar relation holds for the multicomponent case, we gain the following insights. For relative spectral gaps $\Delta_j$ exceeding $4$, the optimal value of $\omega$ is~$0$. In practice, however, these gaps are typically much smaller; for example, in the models above, they range from $6.3\cdot 10^{-6}$ to $3.1\cdot 10^{-2}$.
In this case, the acceleration in convergence provided by the Lagrangian-based metric outweighs the additional computational cost of the ill-conditioned linear systems, resulting in an~optimal value of $\omega$ close to 1. In fact, the main limiting factor in these cases is the missing positive definiteness of $G_{\Phi_k,\omega,j}$ for iterates~$\Phi_k$ that are not sufficiently close to the ground state~$\Phi_*$. Consequently, a~promising approach might be an~adaptive algorithm that chooses a~different $\omega_k$ at each iteration as large as possible while still ensuring that the corresponding linear systems are solvable, similar to the one proposed in \cite{MisS16}.


\section{Conclusion} \label{sec:conclusion}

In this paper, we have presented a~unified geometric and algorithmic framework for the computation of ground states of rotating multicomponent Bose–Einstein condensates. By formulating the problem on an appropriate quotient manifold and employing Riemannian optimization techniques, we have addressed the fundamental difficulties caused by phase invariance and nonlinear coupling inherent to the Gross–Pitaevskii system. Within this setting, we have presented a~class of Riemannain gradient descent methods and analyzed their local convergence properties by deriving convergence rates under general metric assumptions.
These results then serve as the setting to investigate two specific numerical schemes, the energy-adaptive and Lagrangian-based Riemannian gradient descent methods, for which the choice of metric is motivated by the structure of the minimization problem.
Our theoretical analysis clarifies the role of the chosen metrics in shaping the convergence behavior, while numerical experiments support the theoretical findings by confirming the predicted rates and demonstrating the practical effectiveness of the proposed optimization algorithms. Enhanced with an~adaptive step size strategy, they exhibit faster convergence and improved stability, effectively mitigating stagnation effects caused by small spectral gaps. Beyond the present setting, the developed framework paves the way for extending Riemannian optimization methods to more complex quantum systems and other classes of constrained energy minimization problems.

\appendix
\section{Proof of global convergence}\label{app:proof_global}

In order to establish global convergence of the eaRGD method \eqref{eq:eaRGD} for sufficiently small step sizes $\tau_k$, we first present some technical results.

\begin{lemma}\label{lm:estGradRetr}
 For all $\varphibf\in\OB$, $\zbf\in T_\varphibf\OB$, and $\tau \geq 0$, the Riemannian gradient $\grad_a\calE(\varphibf)$ in \eqref{eq:eagradE} can be bounded by $\|\grad_a\calE(\varphibf)\|_{a_{\varphibf}} \!\leq \|\varphibf\|_{a_{\varphibf}}$, and the retraction~\eqref{eq:retraction} satisfies$\!\!$
\begin{equation}\label{eq:estR}
\|\calR_\varphibf(\tau\zbf)-(\varphibf+\tau\zbf)\|_{a_\varphibf}
\leq \frac{\tau^2}{2}\, \|N^{-1}\|_2 \|\zbf\|_{\Lsp}^2 \|\varphibf+\tau\zbf\|_{a_\varphibf}.
\end{equation}
\end{lemma}
\begin{proof}
The estimate for the Riemannian gradient $\grad_a\calE(\varphibf)$ follows from the expression \eqref{eq:eagradE}.
The bound \eqref{eq:estR} can be proved analogously to \mbox{\cite[Prop.~3.11]{AltPS22}} and \cite[Lem.~4.3]{ChenLLZ24}.
\end{proof}

\begin{lemma}\label{lm:estEnergy} 
    Let Assumptions~\textup{\textbf{A1}}--\textup{\textbf{A3}} be fulfilled. For all $\varphibf,\psibf \in \Hsp$, it holds that
    \begin{equation}\label{eq:diffE}
    \calE(\varphibf) - \calE(\psibf)
    = \frac{1}{2}\big(a_\varphibf(\varphibf,\varphibf) - a_\varphibf(\psibf,\psibf)\big)
    - \frac{1}{4}\, \int_\calD (\varphibf \circ \overline{\varphibf} - \psibf \circ \overline{\psibf}) K (\varphibf \circ \overline{\varphibf} - \psibf \circ \overline{\psibf})^T \dx.
    \end{equation}
    In addition, for any $\varphibf \in \OB$ and $\zbf \in T_\varphibf\OB$, we have $\,\calE(\varphibf + \zbf) - \calE(\calR_\varphibf(\zbf)) \geq 0$.
\end{lemma}
\begin{proof}
 For all $\varphibf,\psibf \in \Hsp$, we have
\begin{align*}
\calE(\varphibf) & = \frac{1}{2}a_\varphibf(\varphibf, \varphibf) - \frac{1}{4}\int_\calD(\varphibf \circ \overline{\varphibf})K(\varphibf \circ \overline{\varphibf})^T \dx, \\
\calE(\psibf) & = \frac{1}{2}a_\varphibf(\psibf,\psibf) - \frac{1}{2}\int_\calD(\varphibf \circ \overline{\varphibf})K(\psibf \circ \overline{\psibf})^T \dx + \frac{1}{4} \int_\calD(\psibf \circ \overline{\psibf})K(\psibf \circ \overline{\psibf})^T.
\end{align*}
Combining these relations and using the symmetry of $K$, we obtain \eqref{eq:diffE}.

Furthermore, for any $\varphibf \in \OB$ and $\zbf \in T_\varphibf\OB$, we set $\psibf = \varphibf + \zbf = (\psi_1, \ldots, \psi_p)$ and observe that $\calR_\varphibf(\zbf) = \psibf\,\Sigma$, where $\Sigma \in \DpR$ with
the diagonal entries 
\[
\Sigma_{jj} = N_j^{1/2}\|\psi_j\|_{L^2}^{-1} 
= \big(1 + N_j^{-1}\|z_j\|_{L^2}^2\big)^{-1/2}, \qquad j = 1, \dots, p.
\]
Using  \eqref{eq:energy2}, we then compute
\begin{align*}
    \calE(\varphibf & + \zbf) - \calE(\calR_\varphibf(\zbf))  = \calE(\psibf) - \calE(\psibf\,\Sigma) \\
 & = \frac{1}{2} \left( \|\psibf\|^2_{\sR} -  \|\psibf\,\Sigma\|^2_{\sR} \right)
 + \frac{1}{4}\, \int_{\calD} (\psibf \circ \overline{\psibf})(K-\Sigma^2K\Sigma^2)(\psibf \circ \overline{\psibf})^T \, \dx  \\
 & = \sum_{j=1}^p \int_\calD \frac{1}{2}\big(1\!-\!\Sigma_{jj}^2\big)\big(\|\nablaR\psi_j\|^2+V_j^{\sR}(x)|\psi_j|^2\big)\dx 
    + \frac{1}{4}\sum_{i,j=1}^p(1\! -\! \Sigma_{ii}^2\Sigma_{jj}^2)\kappa_{ij}\int_{\calD} |\psi_i|^2|\psi_j|^2 \dx \geq 0,
\end{align*}
where the last inequality holds due to $0 < \Sigma_{jj} \leq 1$ and $\kappa_{ij} \geq 0$ for $i, j = 1, \dots, p$.
\end{proof}

The following theorem shows that for sufficiently small step sizes, the iterates of the eaRGD method~\eqref{eq:eaRGD} are uniformly bounded and the energy functional $\,\calE$ decays. It can be proved analogously to the non-rotating multicomponent case \cite[Th.~13]{AHPS25}.

\begin{theorem}[Energy decay]\label{thm:EnergyDecay}
\!\!\!\!Let Assumptions~\textup{\textbf{A1}}--\textup{\textbf{A3}} be fulfilled and let $C_K = C_4^4\|K\|_2$. Then there exists $C_0 > 0$ such that for any step size \mbox{$0 < \tau_{\min} \leq \tau_k \leq \tau_{\max} \leq \big(1 + \tfrac{9}{2}C_KC_0^2\big)^{-1} < 1$}, the sequence $\{\varphibf_k\}_{k=0}^\infty\!\subset\! \OB$ generated by the eaRGD method \eqref{eq:eaRGD} has the following pro\-perties:
\[
       {\rm (i)}\; \|\varphibf_k\|_{a_{\varphibf_k}}\leq C_0, 
       \qquad {\rm (ii)} \;\calE(\varphibf_k)-\calE(\varphibf_{k+1})\geq 
       \frac{1}{2}\, \tau_{\min}\, \|\grad_a\calE(\varphibf_k)\|_{a_{\varphibf_k}}^2.
\]
\end{theorem}
We now state a lower bound on the energy decay between successive iterates, showing that the decrease in energy is controlled from below by the squared $a_{\varphibf_k}$-norm of the difference of two consecutive iterates.
\begin{corollary}\label{cor:Energyerror_lowerbound}
There exists a constant $C(N, \tau_{\min}, \tau_{\max}, C_0)$ depending on $N, \tau_{\min}, \tau_{\max}$, and~$C_0$ as defined in Theorem~\textup{\ref{thm:EnergyDecay}} such that 
\begin{equation}\label{eq:estEnergy}
  \calE(\varphibf_k)-\calE(\varphibf_{k+1})  \geq C(N, \tau_{\min}, \tau_{\max}, C_0) \|\varphibf_{k}-\varphibf_{k+1} \|_{a_{\varphibf_k}}^2.
    \end{equation}
\end{corollary}
     
\begin{proof}
 To see this, we first express the difference of $\varphibf_{k+1}$ and $\varphibf_k$ as 
\[
\varphibf_{k+1}-\varphibf_{k} = \calR_{\varphibf_k}(-\tau_k \gbf_k) - (\varphibf_k -\tau_k \gbf_{k}) - \tau_k \gbf_k, 
\]
where $\gbf_k=\grad_a\calE(\varphibf_k)$. Then using Lemma~\ref{lm:estGradRetr} and the relation $\|\varphibf_k\|_{a_{\varphibf_k}} \leq C_0$, we can bound the above difference as
\begin{align*}
    \| \varphibf_{k+1}-\varphibf_{k}\|_{a_{\varphibf_k}} 
    &\leq \frac{\tau_k^2}{2} \|N^{-1}\|_{2} \|\gbf_k\|_{a_{\varphibf_k}}^2 \|\varphibf_k - \tau_k \gbf_k \|_{a_{\varphibf_k}} + \tau_k \|\gbf_k\|_{a_{\varphibf_k}} \\
    &\leq \tau_k \|\gbf_k\|_{a_{\varphibf_k}} \big( 1 + \tau_k \|N^{-1}\|_{2}  \|\gbf_k\|_{a_{\varphibf_k}} \|\varphibf_k - \tau_k \gbf_k \|_{a_{\varphibf_k}} \big) \leq \tau_{\max} ( 1 + C_N)\|\gbf_k\|_{a_{\varphibf_k}},
\end{align*}
with $C_N = \tau_{\max} \|N^{-1}\|_{2}(1+\tau_{\max}) C_0^2$. The last inequality follows from the definition of $\gbf_k$ and the orthogonal projection property. This bound together with that in Theorem~\ref{thm:EnergyDecay} (ii) implies the estimate~\eqref{eq:estEnergy}.
\end{proof}

We are now ready to prove the global convergence result in Theorem~\ref{thm:global_convergence}.

{\it Proof of Theorem~\textup{\ref{thm:global_convergence}}.}

(i) As $\tau_{\max}$ is strictly smaller than $(1+\tfrac{9}{2} C_KC_0^2)^{-1}$, the assumptions of Theorem~\ref{thm:EnergyDecay} are fulfilled. Then the existence of a~limit $\calE_\infty = \lim_{k\to\infty}\calE(\varphibf_k)$ follows from the energy decay and the boundedness of $\,\calE$ from below. 

(ii) By Theorem~\ref{thm:EnergyDecay}(i), the sequence of iterations $\{ \varphibf_k \}_{k=0}^{\infty}$ is uniformly bounded in $\Hsp$-norm. Then the Rellich-Kondrachov theorem guarantees the existence of a~subsequence $\{ \varphibf_{k_l}\}_{l=0}^{\infty}$ that converges to $\varphibf_* \in \OB$ weakly in $\Hsp$ and strongly in $[L^4(\calD,\C)]^p$ and $L$.

Our goal is now to show that $\calA_{\varphibf_{k_l}}^{-1}\varphibf_{k_l}$ converges to $\calA_{\varphibf_*}^{-1}\varphibf_*$ strongly in $\Hsp$. First, for any $k \geq 0$, we obtain that
\begin{equation*}
  c_1\|\calA_{\varphibf_{k}}^{-1}\varphibf_{k}\|_{\Hsp}^2 
  \leq \|\calA_{\varphibf_k}^{-1}\varphibf_k\|_{a_{\varphibf_{k}}}^2 
  = (\varphibf_k,\calA_{\varphibf_k}^{-1}\varphibf_k)_\Lsp \leq \|\varphibf_k\|_\Lsp\|\calA_{\varphibf_k}^{-1}\varphibf_k\|_\Lsp 
  \leq C_2 \trace N \|\calA_{\varphibf_k}^{-1}\varphibf_k\|_\Hsp,
\end{equation*}
 where $C_2$ comes from the Sobolev embedding $\Hsp\hookrightarrow \Lsp$. 
 Therefore, the sequence $\{\calA_{\varphibf_k}^{-1}\varphibf_k\}_{k=0}^\infty$ is uniformly bounded in $\Hsp$-norm. Further, for any $\vbf \in \Hsp$, we have  
\begin{equation*}
\big| a_{\varphibf_*}(\calA_{\varphibf_{k_l}}^{-1}\varphibf_{k_l} - \calA_{\varphibf_*}^{-1}\varphibf_*,\vbf) \big| \leq \big| a_{\varphibf_{k_l}}(\calA_{\varphibf_{k_l}}^{-1}\varphibf_{k_l},\vbf) - a_{\varphibf_*}(\calA_{\varphibf_{k_l}}^{-1}\varphibf_{k_l},\vbf) \big| + \big|(\varphibf_{k_l} - \varphibf_*, \vbf)_\Lsp \big|. 
\end{equation*}
The first term converges to zero as $l\to\infty$ due to \eqref{eq:boundedness} and the boundedness of $\{\calA_{\varphibf_k}^{-1}\varphibf_k\}_{k=0}^\infty$. The second term tends to zero as $l\to\infty$ due to the weak convergence of $\varphibf_{k_l}$ to $\varphibf_*$ in $\Hsp$ and therefore in $\Lsp$. This implies the weak convergence of $\calA_{\varphibf_{k_l}}^{-1}\varphibf_{k_l}$ to $\calA_{\varphibf_*}^{-1}\varphibf_*$ in $\Hsp$ and, as a consequence, the strong convergence in $\Lsp$. Additionally, we can estimate 
\begin{align*}
\big|\|\calA_{\varphibf_{k_l}}^{-1}\varphibf_{k_l}\|_{a_{\varphibf_*}}^2 - \|\calA_{\varphibf_*}^{-1}\varphibf_*\|_{a_{\varphibf_*}}^2 \big| \leq & \big|\|\calA_{\varphibf_{k_l}}^{-1}\varphibf_{k_l}\|_{a_{\varphibf_*}}^2 - \|\calA_{\varphibf_{k_l}}^{-1}\varphibf_{k_l}\|_{a_{\varphibf_{k_l}}}^2 \big| \\
& + \big| (\varphibf_{k_l}, \calA_{\varphibf_{k_l}}^{-1}\varphibf_{k_l} - \calA_{\varphibf_*}^{-1}\varphibf_*)_\Lsp \big| + \big| (\varphibf_{k_l} - \varphibf_*,\calA_{\varphibf_*}^{-1}\varphibf_*)_\Lsp\big|, 
\end{align*}
 where all three terms go to zero as $l\to\infty$ due to the corresponding strong convergences in~$\Lsp$ and boundedness. Therefore, $\calA_{\varphibf_{k_l}}^{-1}\varphibf_{k_l}$ converges strongly in $\|\cdot\|_{a_{\varphibf_*}}$ 
 to $\calA_{\varphibf_*}^{-1}\varphibf_*$. Moreover, for $\Lambda_{k_l}= N \out{\varphibf_{k_l}}{\calA^{-1}_{\varphibf_{k_l}} \varphibf_{k_l} }^{-1}$ and $\Lambda_* = N \out{\varphibf_*}{\calA^{-1}_{\varphibf_*} \varphibf_*}^{-1}$, this also implies that $\lim_{l\to\infty}\Lambda_{k_l} = \Lambda_*$.

Introducing $\gbf_{k_l} = \grad_a\calE(\varphibf_{k_l}) = \varphibf_{k_l} - \calA_{\varphibf_{k_l}}^{-1}\varphibf_{k_l}\Lambda_{k_l}$ and $\gbf_* = \grad_a\calE(\varphibf_*) = \varphibf_* - \calA_{\varphibf_*}^{-1}\varphibf_*\Lambda_*$, we notice that $\gbf_{k_l}$ converges weakly in $\Hsp$ to $\gbf_*$. As we know by Theorem~\ref{thm:EnergyDecay}(ii) and the convergence of $\,\calE(\varphibf_k)$ that 
$0 = \lim_{l\to\infty}\|\gbf_{k_l}\|_\Hsp \geq \|\gbf_*\|_\Hsp$,
where the lower bound follows from the weak lower semicontinuity of the norm,
$\gbf_{k_l}$ converges strongly in $\Hsp$ to $\gbf_* = 0$. 
Therefore, $\varphibf_{k_l} = \gbf_{k_l} + \calA_{\varphibf_{k_l}}^{-1}\varphibf_{k_l}\Lambda_{k_l}$ 
converges strongly as well and, due to the continuity of $\,\calE$, we  have $\calE_\infty=\lim_{k\to\infty} \calE(\varphibf_{k_l}) = \calE(\varphibf_*)$.

Finally, applying $\calA_{\varphibf_*}$ to both sides of 
$0 = \gbf_* = \varphibf_* - \calA_{\varphibf_*}^{-1}\varphibf_* N\out{\varphibf_*}{\calA_{\varphibf_*}^{-1}\varphibf_*}^{-1}$, we obtain that $\calA_{\varphibf_*}\varphibf_* = \varphibf_* N\out{\varphibf_*}{\calA_{\varphibf_*}^{-1}\varphibf_*}^{-1}$ and $\coout{\calA_{\varphibf_*}\varphibf_*}{\varphibf_*}N^{-1}=N \out{\varphibf_*}{\calA^{-1}_{\varphibf_*} \varphibf_*}^{-1}=\Lambda_*$. Thus, $\varphibf_*$ is indeed a constraint critical point of $\,\calE$ with the Lagrange multiplier $\Lambda_*$.

(iii) Since the limit energy $\calE_\infty$ is unique, all accumulation points of $\{\varphibf_k\}_{k=0}^\infty$ have the same energy level. Now assume that one such an~accumulation point, say $\varphibf_*$, is a locally quasi-unique ground state to \eqref{eq:min}. Then all other accumulation points must be ground states. To account for the phase shifts, we work on the quotient manifold $\,\Seg$, the space of equivalence classes in $\OB$, and later on translate our result back to the oblique manifold $\OB$. To this end, we define the distance between the equivalence classes $[\varphibf]$ and $[\varthetabf]$ on $\,\Seg$ by
\begin{align*}
\dist([\varphibf],[\varthetabf]) := \inf_{(\widehat{\varphibf},\widehat{\varthetabf}) \in [\varphibf] \times [\varthetabf] } \| \widehat{\varphibf} - \widehat{\varthetabf}\|_{\Hsp}. 
\end{align*}
As $[\varphibf_*]$ is a strict local minimizer to \eqref{eq:min_quotient}, there exists a~$\delta > 0$ such that $\dist([\varphibf_*], [\varthetabf_*]) > \delta$ for every other ground state $[\varthetabf_*] \in \Seg$. Hence, for any $0 < \epsilon < \delta$, there are only finitely many $\varphibf_k$ such that $\frac{\epsilon}{2} \leq \dist([\varphibf_*], [\varphibf_k]) \leq \epsilon$. Otherwise, we could select a weakly converging subsequence among these elements $\varphibf_k$ and repeat the same arguments as in (ii) to prove that this subsequence is strongly converging in $\Hsp$ to a ground state $\varphibf_*^\prime$ with $\dist([\varphibf_*],[\varphibf_*^\prime]) \leq \epsilon < \delta$. Further, by Corollary~\ref{cor:Energyerror_lowerbound}, we obtain that $\dist([\varphibf_k],[\varphibf_{k+1}]) \leq \|\varphibf_k - \varphibf_{k+1}\|_\Hsp \rightarrow 0$ as $k \to \infty$. Hence, there are only finitely many $\varphibf_k$ such that $\dist([\varphibf_k],[\varphibf_*]) \geq \epsilon$. This implies that $[\varphibf_k]$ converges to $[\varphibf_*]$. With the diamagnetic inequality \cite[Th.~7.21]{LieL01}, we finally obtain that
\[
\| \, |\varphibf_k| - |\varphibf_*| \, \|_{\Hsp} \le 
\inf_{ \Theta_1,\Theta_2 \in \DpS } \| \varphibf_k\Theta_1 - \varphibf_*\Theta_2 \|_{\Hsp}
= \dist([\varphibf_k],[\varphibf_*]) \rightarrow 0 \quad \text{as } \;k \to \infty.
\]
This completes the proof.  \hfill\proofbox


\bibliographystyle{siamplain}
\bibliography{ref}
\end{document}